\documentclass[12pt]{article}
\usepackage{amssymb,amsmath,graphicx,amsthm,mathrsfs,bm}
\usepackage{color}
\newtheorem{Theorem}{Theorem}[section]
\newtheorem{Lemma}[Theorem]{Lemma}
\newtheorem{Proposition}[Theorem]{Proposition}

\newtheorem{Corollary}[Theorem]{Corollary}

\newtheorem{Remark}[Theorem]{Remark}
\numberwithin{equation}{section}

\newcommand{\lc}
{\mathrel{\raise2pt\hbox{${\mathop<\limits_{\raise1pt\hbox
{\mbox{$\sim$}}}}$}}}

\newcommand{\gc}
{\mathrel{\raise2pt\hbox{${\mathop>\limits_{\raise1pt\hbox{\mbox{$\sim$}}}}$}}}

\newcommand{\ec}
{\mathrel{\raise2pt\hbox{${\mathop=\limits_{\raise1pt\hbox{\mbox{$\sim$}}}}$}}}

\begin{document}

\title{Minimal time  control of exact synchronization for parabolic systems }

\author{Lijuan Wang\thanks{School of
Mathematics and Statistics, Computational Science Hubei Key Laboratory,
Wuhan University, Wuhan 430072, China;
e-mail: ljwang.math@whu.edu.cn. This work was supported by the
National Natural Science Foundation of China under grant 11771344.}
\quad and \quad Qishu Yan
\thanks{Corresponding author. School of Science,
Hebei University of Technology, Tianjin 300400, China. e-mail:
yanqishu@whu.edu.cn.This work was supported by the
National Natural Science Foundation of China under grant 11701138.}}

\date{}
\maketitle
\begin{abstract}
This paper studies  a kind of minimal time control problems related to the exact
synchronization for  a controlled  linear system of parabolic equations.
Each problem depends on two parameters: the bound of controls
and the initial state.
The purpose of such a problem is to find a control (from a constraint set) synchronizing components of the corresponding solution vector for the controlled system in the shortest time.
In this paper,  we
build up a necessary and sufficient condition for the optimal
time and the optimal control; we also obtain how the existence  of optimal controls
depends on the above mentioned two parameters.
\end{abstract}

\medskip


\medskip


\section{Introduction}

\subsection{Synchronization and control}

Synchronization is a widespread natural phenomenon. For instance, thousands of fireflies may twinkle at the same time;
field crickets give out a unanimous cry; audiences in the theater can applaud with a rhythmic beat;
and pacemaker cells of the heart function simultaneously (see \cite{Li-Rao:1}, \cite{Li-Rao:2},
\cite{Li-Rao:3} and the references therein). The phenomenon of synchronization was first observed by Huygens in
1665 (see \cite{Huygens}). The theoretical studies on synchronization phenomena from
mathematical perspective were started by  Wiener in the 1950s (see \cite{Wiener}).
The previous studies focused on systems described by ODEs such as
\begin{equation*}
\frac{d \bm X_i}{dt}=\bm f(\bm X_i,t)+\sum_{j=1}^N A_{ij}\bm X_j,\;\;\;i=1,\dots,N,
\end{equation*}
where $\bm X_i (i=1,\dots,N)$ are $n$-dimensional state vectors, $A_{ij} (1\leq i,j\leq N)$ are
$n\times n$ coupling matrices, and $\bm f(\bm X,t)$ is an $n$-dimensional function independent of
$i$. If for any given initial data $\bm X_i(0)=\bm X_i^{(0)} (1\leq i\leq N)$,
the solution $(\bm X_1,\dots,\bm X_N)$ to the system satisfies
\begin{equation*}
\bm X_i(t)-\bm X_j(t)\rightarrow \bm 0\;\;(1\leq i,j\leq N)\;\;\mbox{as}\;\;t\rightarrow +\infty,
\end{equation*}
then we say that the system possesses the synchronization (see \cite{Wu}, \cite{Li-Rao:3}, \cite{Li-Yang}
 and the references therein ).

Mathematically, the exact synchronization for a controlled system is to ask for a control
 so that the difference of any two  components of  the corresponding solution to the system
 (with an initial state) takes value zero at a fixed time and remains the value zero
 after the aforementioned fixed time. The exact synchronization in the PDEs case
 was first studied for a coupled system of wave equations both for the higher-dimensional case in the
 framework of weak solutions in \cite{Li-Rao:4}, \cite{Li-Rao:1} and \cite{Li-Rao:5}, and
 for the one-dimensional case in the framework of classical solutions in \cite{Hu} and \cite{Li-Rao:2}.
 Recently, Pontryagin's maximum principle of optimal control problems for the exact synchronization of parabolic
systems was studied in \cite{Wang-Yan}.

Minimal time  control for the exact synchronization of controlled systems is another interesting topic. It is to ask for a control from a constraint set so that the difference
of any two  components of  the corresponding solution to the system (with an initial state) takes value zero in the shortest time and remains the value zero after the aforementioned shortest time. It is a kind of time optimal control problem. To the best of our knowledge, such problem has not been touched upon. This paper studies a minimal time  control problem for the exact synchronization of some parabolic systems.

\subsection{Formulation of the problem}

This subsection formulates the problem  studied in this paper.
We begin with introducing the controlled system. Let
 $\Omega\subseteq \mathbb R^d$ (with $d\geq 1$) be a bounded domain
with a $C^2$ boundary $\partial \Omega$. Let $\omega\subseteq \Omega$ be an open and nonempty subset with its characteristic function
$\chi_\omega$.
Let $A\triangleq (a_{ij})_{1\leq i, j\leq n}\in \mathbb R^{n\times n}$ and
$B\triangleq (b_{ij})_{1\leq i\leq n, 1\leq j\leq m}\in \mathbb R^{n\times m}$ be two constant matrices,
where $n\geq 2$ and $m\geq 1$. Let $\bm y_0\in L^2(\Omega)^n$. Consider the controlled linear parabolic system:
\begin{equation}\label{Intro:1}
\left\{
\begin{array}{lll}
\bm y_t-\Delta \bm y+A\bm y=\chi_\omega B\bm u&\mbox{in}&\Omega\times (0,+\infty),\\
\bm y=\bm 0&\mbox{on}&\partial \Omega\times(0,+\infty),\\
\bm y(0)=\bm y_0&\mbox{in}&\Omega,
\end{array}
\right.
\end{equation}
where $\bm u\in L^2(0,+\infty; L^2(\Omega)^m)$ is a control.
Write
$$\bm y(t;\bm y_0,\bm u)=(y_1(t;\bm y_0,\bm u),y_2(t;\bm y_0,\bm u),\dots,y_n(t;\bm y_0,\bm u))^\top
$$
 for the solution of the system (\ref{Intro:1}).
 (Here and throughout this paper, we denote the transposition of a matrix $J$ by $J^\top$.)
   It is well known that for each $T>0$, $\bm y(\cdot;\bm y_0,\bm u)\in W^{1,2}(0,T;H^{-1}(\Omega)^n)
\cap L^2(0,T;H_0^1(\Omega)^n)\subseteq C([0,T];L^2(\Omega)^n)$.
We will treat this solution
as a function from $[0,+\infty)$
to $L^2(\Omega)^n$.

We next define  control constraint set $\mathcal{U}_M$ (with $M>0$) and the target set $S$ as follows:
\begin{equation*}
\mathcal{U}_M\triangleq \left\{\bm u\in L^2(0,+\infty;L^2(\Omega)^m)\; :\;
\|\bm u\|_{L^2(0,+\infty;L^2(\Omega)^m)}\leq M\right\};
\end{equation*}
\begin{equation*}
S\triangleq \{(y_1,y_2,\dots,y_n)^\top\in L^2(\Omega)^n: y_1=y_2=\dots=y_n\}.
\end{equation*}

Given $M>0$, $\bm y_0\in L^2(\Omega)^n$, we define  the  minimal time  control problem
$(TP)_M^{\bm y_0}$:
\begin{equation*}
T(M,\bm y_0)\triangleq\mbox{inf}_{\bm u\in \mathcal{U}_M}\{T\geq 0\;:\;\bm u(\cdot)=\bm 0
\;\;\mbox{and}\;\;\bm y(\cdot;\bm y_0,\bm u)\in S\;\mbox{over}\; [T,+\infty)\}.
\end{equation*}
About Problem $(TP)_M^{\bm y_0}$, several notes are given in order:
\begin{description}
\item[($a_1$)]
    We call $T(M,\bm y_0)$ the optimal time; we call $\bm u\in \mathcal{U}_M$
    an admissible control if there is $T\geq 0$ so that  $\bm u(\cdot)=\bm 0$ and
    $\bm y(\cdot;\bm y_0,\bm u)\in S$ over $[T,+\infty)$; we call $\bm u^*\in \mathcal{U}_M$
    an optimal control if $T(M,\bm y_0)<+\infty, \bm u^*(\cdot)=\bm 0$ and
    $\bm y(\cdot;\bm y_0,\bm u^*)\in S$ over $[T(M,\bm y_0),+\infty)$; we agree that
    $T(M,\bm y_0)=+\infty$ if the problem $(TP)_M^{\bm y_0}$ has no admissible control.

 \item[($a_2$)] One can easily check that if $\bm y=(y_1,y_2,\dots,y_n)^\top\in L^2(\Omega)^n$,
 then $\bm y\in S$ if and only if $D\bm y=\bm 0$, where and throughout this paper,
     \begin{equation}\label{wang1.3}
D\triangleq\left(
          \begin{array}{ccccc}
            1 & -1 & 0  &\cdots  & 0 \\
            0 & 1  & -1 &\cdots  & 0\\
            \vdots & \vdots & \ddots & \ddots &\vdots \\
            0 & 0  &    \cdots&1& -1 \\
          \end{array}
        \right)_{(n-1)\times n}.
\end{equation}

\item[$(a_3)$] Differing from a general minimal time control problem,
our problem here is to ask for a control (from the constraint set)
not only driving the corresponding solution to the target $S$ at the shortest time,
but also remaining the solution in $S$  after the shortest time with the null control.
This arises from the characteristic of the exact synchronization.
When the target is an equilibrium solution of the system with the null control,
this can be done by taking the null control after the shortest time. However,
the elements in $S$ may not be equilibrium solutions. Thus, we need some reasonable assumptions to fit it.

    \item[($a_4$)] Two concepts related to this problem are the null controllability
    (see \cite{Fernandez}) and
   the exact synchronization (see \cite{Li-Rao:4}). Let us recall them: First,
   the system (\ref{Intro:1}) is said to be null controllable at time $T$, if for any $\bm y_0\in L^2(\Omega)^n$,
there exists a control $\bm u\in L^2(0,+\infty;L^2(\Omega)^m)$, with $\bm u=\bm 0$ over $(T,+\infty)$,
so that $\bm y(T;\bm y_0,\bm u)=\bm 0$. Second, the system (\ref{Intro:1}) is said to be
exactly synchronizable at time $T$, if for any $\bm y_0\in L^2(\Omega)^n$,
there exists a control $\bm u\in L^2(0,+\infty;L^2(\Omega)^m)$, with $\bm u=\bm 0$ over $(T,+\infty)$,
so that
\begin{equation*}
y_1(t;\bm y_0,\bm u)=y_2(t;\bm y_0,\bm u)=\cdots=y_n(t;\bm y_0,\bm  u)\;\;\mbox{for all}\;\;t\geq T.
\end{equation*}
 Mathematically, the exact synchronization is weaker than the
null controllability.

\end{description}

\subsection {Aim, motivation and hypotheses}

{\bf Aim} First, we will answer the question: Given $(\bm y_0,M)\in L^2(\Omega)^n\times (0,+\infty)$, does the problem $(TP)_M^{\bm y_0}$ has an optimal control? Second,
we are going to   characterize the optimal time and the optimal control
to the problem $(TP)_M^{\bm y_0}$.

\vskip 5pt

\noindent {\bf Motivation} As we have explained before, the minimal time  control problem
for the exact synchronization of controlled systems is an interesting topic
and has not been touched upon. In the problem $(TP)_M^{\bm y_0}$,
the optimal time and the optimal control are two of the most important quantities.
It is not an easy job to characterize them.
In most papers concerning time optimal control problems,
people can only provide a necessary condition for the optimal control,
i.e., Pontryagin's maximum principle (see, for instance, \cite{Barbu}, \cite{Kunisch-Wang1}
and \cite{Li-Yong}).
In \cite{Wang-Zuazua}  and \cite{Kunisch}, the authors gave characteristics for the optimal time
and the optimal control for a minimal time control problem,
where the controlled system is the heat equation and the target is the origin (in the state space)
which is an equilibrium solution of the heat equation.
Since the target set $S$ differs from the target in \cite{Wang-Zuazua} and \cite{Kunisch},
we cannot directly use the way in \cite{Wang-Zuazua} and  \cite{Kunisch}
to get a necessary and sufficient condition for the optimal time and the optimal control for the problem $(TP)_M^{\bm y_0}$.
These motivate us to do this research.

\vskip 5pt

\noindent{\bf Hypotheses} Our main theorems are based on one of the following  two hypotheses.
\begin{itemize}
  \item[$(H_1)$] The pair $(A,B)$ satisfies that
  \begin{equation}\label{wang1.5}
  \sum_{\ell=1}^n a_{i\ell}=\sum_{\ell=1}^n a_{j\ell}\;\;\mbox{for all}\;\;i, j\in \{1,2,\dots,n\};
  \end{equation}
  and that
    \begin{equation}\label{wang1.5-1}
  \mbox{rank}(DB,DAB,\dots,DA^{n-2}B)= n-1.
  \end{equation}
  (Recall that $D$ is given by (\ref{wang1.3}).)
  \item[$(H_2)$] The pair $(A,B)$ satisfies that
  \begin{equation}\label{wyuanyuan1.7}
  \sum_{\ell=1}^n a_{i_0 \ell}\neq \sum_{\ell=1}^n a_{j_0 \ell}\;\;\mbox{for some}\;\;
  i_0, j_0\in \{1,2,\dots,n\};
  \end{equation}
  and that
    \begin{equation}\label{wyuanyuan1.8}
  \mbox{rank}(B,AB,\dots,A^{n-1}B)=n.
  \end{equation}
\end{itemize}
Several remarks on these hypotheses are given in order.
\begin{description}
\item[$(b_1)$] One can easily see that $(H_1)$ differs from $(H_2)$.
\item[$(b_2)$] We call (\ref{wang1.5}) (or (\ref{wyuanyuan1.7})) the  row condition; and call
(\ref{wang1.5-1}) (or (\ref{wyuanyuan1.8})) the  rank condition.
\item[($b_3$)] (\ref{wang1.5}) is equivalent to that
(see \cite{Li-Rao:3} and \cite{Li-Rao-Wei}) there exists a unique matrix
$\widetilde A\in \mathbb R^{(n-1)\times (n-1)}$ so that
\begin{equation}\label{wang1.6}
DA=\widetilde A D.
\end{equation}

\item[($b_4$)] There is a pair $(A,B)$ satisfying $(H_1)$. For example,
\begin{equation*}
A=\left(
              \begin{array}{cc}
                1 & 0 \\
                0.5 & 0.5 \\
              \end{array}
            \right),
B=\left(
              \begin{array}{c}
                0 \\
                1 \\
              \end{array}
            \right).
\end{equation*}
\item[($b_5$)] There is a pair $(A,B)$ satisfying  $(H_2)$. For example,
\begin{equation*}
A=\left(
              \begin{array}{cc}
                1 & 2 \\
                3 & 4 \\
              \end{array}
            \right),
B=\left(
              \begin{array}{c}
                0 \\
                1 \\
              \end{array}
            \right).
\end{equation*}
\item[($b_6$)] We proved in this paper
that the system (\ref{Intro:1}) is exactly synchronizable at time $T$ if and only if
$(A, B)$ satisfies either $(H_1)$ or $(H_2)$ (see Remark~\ref{yuanyuanremark2.2}).
Hence, our main theorems are based on $(H_1)$ and $(H_2)$.  In \cite{Wang-Yan}, another necessary and
sufficient condition for the exact synchronization of the system (\ref{Intro:1}) has been proved.
The difference between \cite{Wang-Yan} and this paper is as follows: In \cite{Wang-Yan}, the research is carried out  from the perspective of rank condition: rank($B,AB,\dots,A^{n-1}B)=n$ or
rank($B,AB,\dots,A^{n-1}B)\not=n$, while in this paper, it is carried out from
the perspective of the row condition:  (\ref{wang1.5}) or (\ref{wyuanyuan1.7}).
It deserves mentioning that necessary and sufficient conditions on the
synchronization was obtained for controlled systems of wave equations in \cite{Li-Rao:4}.

\end{description}

\subsection{Plan of the paper}
In section 2, we present the main results of this paper.
 In section 3, we give some preliminaries which contain  the exact synchronization for the controlled system of heat equations and  transformations for the minimal time control problem.  In section
4, we present some properties on a minimal norm control problem.
In the final section, we prove the main results of this paper.
Throughout this paper, $C(\cdot)$ denotes a generic positive constant, which depends on
what are enclosed in the bracket.

\section{Main results}
Our main results will be given by  two theorems. To state them, we need to 
introduce one kind of minimal norm control problems and two kinds of functionals under either $(H_1)$ or $(H_2)$.

\noindent {\bf Minimal norm control problem}
Given $T>0$ and $\bm y_0\in L^2(\Omega)^n$, define the minimal norm control problem
 $(NP)_T^{\bm y_0}$ in the following manner:
\begin{equation*}
N(T,\bm y_0)\triangleq \inf\{\|\bm v\|_{L^2(0,+\infty;L^2(\Omega)^m)}: \bm v(\cdot)=\bm 0
\;\,\mbox{and}\;\,\bm y(\cdot;\bm y_0,\bm v)\in S\;\,\mbox{over}\,\;[T,+\infty)\}.
\end{equation*}

Several notes on the problem  $(NP)_T^{\bm y_0}$ are given in order.

\begin{description}
\item[($c_1$)] We call $N(T,\bm y_0)$  the minimal norm;
we call $\bm v\in L^2(0,+\infty;L^2(\Omega)^m)$ an admissible control if $\bm v(\cdot)=\bm 0$ and
 $\bm y(\cdot;\bm y_0,\bm v)\in S$ over $[T, +\infty)$; we call a function $\bm v^*$
 an optimal control if it is admissible and satisfies that $\|\bm v^*\|_{L^2(0,+\infty;L^2(\Omega)^m)}=N(T,\bm y_0)$.

 \item[($c_2$)] Given $\bm y_0\in L^2(\Omega)^n$, we can  treat $N(\cdot,\bm y_0)$
 as a function of the time variable. We have proved  that if either $(H_1)$ or $(H_2)$ holds,
 then for each $\bm y_0\in L^2(\Omega)^n$,  $\lim _{T\rightarrow+\infty}N(T,\bm y_0)$
exists (see Proposition~\ref{Norm:7}). Thus, under either $(H_1)$ or $(H_2)$,  we can let
\begin{equation}\label{Intro:9}
M(\bm y_0)\triangleq \lim_{T\rightarrow +\infty} N(T,\bm y_0)\;\;\mbox{for each}\;\;\bm y_0\in L^2(\Omega)^n.
\end{equation}

\item[($c_3$)] If either $(H_1)$ or $(H_2)$ holds, then for any $T>0$ and $\bm y_0\in L^2(\Omega)^n$, the problem $(NP)_T^{\bm y_0}$ has a unique optimal control
    (see Theorem \ref{Norm:6}).

\end{description}

\noindent {\bf Two auxiliary functionals }  The first functional is built up (under the assumption $(H_1)$)  in the following manner:
 Recall the note $(b_3)$ for the matrix $\widetilde A$.
Let  $T>0$ and let $\bm y_0\in L^2(\Omega)^n$. Write  $\bm \psi(\cdot;T,\bm \psi_T)$,
with $\bm \psi_T\in L^2(\Omega)^{n-1}$,
for the solution to the  system:
\begin{equation}\label{Intro:11}
\left\{
\begin{array}{lll}
\bm \psi_t+\Delta \bm \psi-\widetilde A^\top \bm \psi=\bm 0&\mbox{in}&\Omega\times (0,T),\\
\bm \psi=\bm 0&\mbox{on}&\partial\Omega\times (0,T),
\end{array}\right.
\end{equation}
with the initial condition $\bm \psi(T)=\bm \psi_T$.
 Construct two subspaces:
\begin{equation}\label{Intro:12}
X_{T,1}\triangleq\{\chi_\omega B^\top D^\top \bm \psi(\cdot;T,\bm \psi_T): \bm \psi_T\in L^2(\Omega)^{n-1}\}
\;\;\mbox{and}\;\;Y_{T,1}\triangleq \overline{X_{T,1}}^{\|\cdot\|_{L^2(0,T;L^2(\omega)^m)}}.
\end{equation}
We can characterize elements in the space  $Y_{T,1}$  (see (i) of Lemma~\ref{Norm:2}).
In fact, each element in $Y_{T,1}$  can be expressed as
$\chi_\omega B^\top D^\top \bm \psi$, where $\bm \psi\in C([0,T);L^2(\Omega)^{n-1})$ solves (\ref{Intro:11})
and satisfies that
$\chi_\omega B^\top D^\top \bm \psi(\cdot)=\lim_{i\rightarrow +\infty}\chi_\omega B^\top D^\top \bm \psi(\cdot;T,\bm z_i)$
for some sequence $\{\bm z_i\}_{i\geq 1}\subseteq L^2(\Omega)^{n-1}$,
where the limit is taken in  $L^2(0,T;L^2(\omega)^m)$.
Define the first  functional $J_{T,1}^{\bm y_0}: Y_{T,1}\rightarrow \mathbb{R}$ by
\begin{equation}\label{Intro:13}
J_{T,1}^{\bm y_0}(\chi_\omega B^\top D^\top \bm \psi)\triangleq
\frac{1}{2}\int_0^T\|\chi_\omega B^\top D^\top \bm \psi\|_{L^2(\omega)^m}^2\mathrm dt
+\langle \bm \psi(0), D\bm y_0\rangle_{L^2(\Omega)^{n-1}}
\end{equation}
for each $\chi_\omega B^\top D^\top \bm \psi\in Y_{T,1}$.

The second functional  is defined (under the assumption $(H_2)$) in the following manner:
Let $T>0$ and let $\bm y_0\in L^2(\Omega)^n$. Write $\bm \varphi(\cdot;T,\bm \varphi_T)$, with  $\bm \varphi_T\in L^2(\Omega)^n$,
for the solution to the  system:
\begin{equation}\label{Intro:14}
\left\{
\begin{array}{lll}
\bm \varphi_t+\Delta \bm \varphi-A^\top \bm \varphi=\bm 0&\mbox{in}&\Omega\times (0,T),\\
\bm \varphi=\bm 0&\mbox{on}&\partial\Omega\times (0,T),
\end{array}\right.
\end{equation}
with the initial condition $\bm \varphi(T)=\bm \varphi_T$.
Build up two subspaces:
\begin{equation}\label{Intro:15}
X_{T,2}\triangleq\{\chi_\omega B^\top \bm \varphi(\cdot;T,\bm \varphi_T): \bm \varphi_T\in L^2(\Omega)^n\}
\;\;\mbox{and}\;\;Y_{T,2}\triangleq \overline{X_{T,2}}^{\|\cdot\|_{L^2(0,T;L^2(\omega)^m)}}.
\end{equation}
We can also characterize elements in the space  $Y_{T,2}$  (see (ii) of Lemma~\ref{Norm:2}).
Indeed,  each element in $Y_{T,2}$ can be expressed as
$\chi_\omega B^\top  \bm \varphi$, where $\bm \varphi\in C([0,T);L^2(\Omega)^{n})$ solves (\ref{Intro:14}) and satisfies that
 $\chi_\omega B^\top  \bm \varphi(\cdot)=\lim_{i\rightarrow +\infty}\chi_\omega B^\top \bm \varphi(\cdot;T,\bm z_i)$
 for some sequence $\{\bm z_i\}_{i\geq 1}\subseteq L^2(\Omega)^{n}$, where the limit is taken in  $L^2(0,T;L^2(\omega)^m)$.
Define the second functional  $J_{T,2}^{\bm y_0}: Y_{T,2}\rightarrow \mathbb{R}$ by
\begin{equation}\label{Intro:16}
J_{T,2}^{\bm y_0}(\chi_\omega B^\top \bm \varphi)\triangleq
 \frac{1}{2}\int_0^T\|\chi_\omega B^\top \bm \varphi\|_{L^2(\omega)^m}^2\mathrm dt
+\langle \bm \varphi(0), \bm y_0\rangle_{L^2(\Omega)^n}
\end{equation}
for each $\chi_\omega B^\top \bm \varphi\in Y_{T,2}$.

Several notes on these two functionals are given in order.
\begin{description}
\item[($d_1$)] The functional $J_{T,1}^{\bm y_0}$
has the following properties: First, it
is well defined on $Y_{T,1}$ (see (i) of Corollary~\ref{20180205-1}); Second, it
has a unique nontrivial minimizer  in $Y_{T,1}$
when $\bm y_0\not\in S$ (see (i) of Lemma~\ref{Norm:5}).
\item[($d_2$)] The functional $J_{T,2}^{\bm y_0}$ has the following properties:
First, it is well defined on $Y_{T,2}$
(see (ii) of Corollary~\ref{20180205-1});
Second, it has a unique nontrivial minimizer in $Y_{T,2}$  when $\bm y_0\not=\bm 0$ (see (ii) of Lemma~\ref{Norm:5}).
\end{description}

Recall (\ref{Intro:9}),
(\ref{wang1.3}), (\ref{Intro:12}), (\ref{Intro:13}),
(\ref{Intro:15}),
(\ref{Intro:16}), and notes $(d_1)$ and $(d_2)$. The  main theorems  of this paper are as follows:
\begin{Theorem}\label{Intro:17-1}
Suppose that $(H_1)$ holds. Let  $\bm y_0\in L^2(\Omega)^n$  and let $M>0$.
The following conclusions are true:
\begin{itemize}
  \item[$(i)$] If  $\bm y_0\in S$, then $(TP)_M^{\bm y_0}$
  has the unique optimal control $\bm 0$ (while $0$ is the optimal time);
    If $\bm y_0\not\in S$ and
  $M\leq M(\bm y_0)$, then $(TP)_M^{\bm y_0}$ has no optimal control;
  If  $\bm y_0\not\in S$ and
  $M>M(\bm y_0)$, then $(TP)_M^{\bm y_0}$ has a unique optimal control which is nontrivial.
  \item[$(ii)$]  If $\bm y_0\not\in S$ and $M>M(\bm y_0)$, then
    $T^*$ and $\bm u^*$ are the optimal time and the optimal control to
$(TP)_M^{\bm y_0}$ if and only if
  \begin{equation}\label{Intro:18}
  M=\Big(\int_0^{T^*} \|\chi_\omega B^\top D^\top \bm \psi^*(t)\|_{L^2(\omega)^m}^2\mathrm dt\Big)^{\frac{1}{2}}
  \end{equation}
 and
 \begin{equation}\label{Intro:19}
 \bm u^*(t)\triangleq \left\{
 \begin{array}{ll}
 \chi_\omega B^\top D^\top \bm \psi^*(t),&t\in (0,T^*),\\
 \bm 0,&t\geq T^*,
 \end{array}\right.
 \end{equation}
 where $\chi_\omega B^\top D^\top \bm \psi^*$,
  with $\bm \psi^*\in C([0,T^*);L^2(\Omega)^{n-1})$ solving (\ref{Intro:11}),
  is the unique minimizer of $J_{T^*,1}^{\bm y_0}$ over $Y_{T^*,1}$.
\end{itemize}
\end{Theorem}

\begin{Theorem}\label{Yuanyuantheorem2.2}
Suppose that $(H_2)$ holds. Let $\bm y_0\in L^2(\Omega)^n$  and  let $M>0$. The following conclusions are true:
\begin{itemize}
  \item[$(i)$]
 If $\bm y_0=\bm 0$, then $(TP)_M^{\bm y_0}$
  has the unique optimal control $\bm 0$ (while $0$ is the optimal time);
 If $\bm y_0\neq\bm 0$ and $M\leq M(\bm y_0)$, then $(TP)_M^{\bm y_0}$ has no optimal control;
  If $\bm y_0\neq\bm 0$ and $M>M(\bm y_0)$, then $(TP)_M^{\bm y_0}$ has a unique optimal control which is nontrivial.
  \item[$(ii)$] If $\bm y_0\not=\bm 0$ and $M>M(\bm y_0)$, then
    $T^*$ and $\bm u^*$ are the optimal time and the optimal control to
$(TP)_M^{\bm y_0}$ if and only if
  \begin{equation*}
  M=\Big(\int_0^{T^*} \|\chi_\omega B^\top \bm \varphi^*(t)\|_{L^2(\omega)^m}^2\mathrm dt
  \Big)^{\frac{1}{2}}
  \end{equation*}
 and
 \begin{equation*}
 \bm u^*(t)\triangleq \left\{
 \begin{array}{ll}
 \chi_\omega B^\top \bm \varphi^*(t),&t\in (0,T^*),\\
 \bm 0,&t\geq T^*,
 \end{array}\right.
 \end{equation*}
 where $\chi_\omega B^\top \bm \varphi^*$,
  with $\bm \varphi^*\in C([0,T^*);L^2(\Omega)^{n})$ solving (\ref{Intro:14}),
  is the unique minimizer of $J_{T^*,2}^{\bm y_0}$
 over $Y_{T^*,2}$.
  \end{itemize}
\end{Theorem}

Several notes on Theorems~\ref{Intro:17-1} and \ref{Yuanyuantheorem2.2} are given in order.
\begin{Remark}\label{Intro:19-1}
\begin{itemize}
\item[$(a)$] The conclusion $(i)$  in Theorem~\ref{Intro:17-1} (or Theorem~\ref{Yuanyuantheorem2.2}) shows how the existence of  optimal controls to $(TP)_M^{\bm y_0}$ depends on   $(M,\bm y_0)\in (0,+\infty)\times
    L^2(\Omega)^n$.
  In \cite{Wang-Zhang}, the authors studied how the bang-bang property of a
  minimal time control problem
  depends on the pair  $(M,\bm y_0)$.
  The target set in \cite{Wang-Zhang} is an equilibrium solution of the system with the null control, while in our paper, the target set $S$ may contain non-equilibrium solutions of the system with the null control.
  \item[$(b)$] The conclusion $(ii)$  in Theorem~\ref{Intro:17-1} (or Theorem \ref{Yuanyuantheorem2.2})
  gives characteristics of the optimal time and the optimal control via the minimizer of
  a given functional, under the assumption $(H_1)$ (or $(H_2)$).
     \item[$(c)$]
   By (ii) in Theorem~\ref{Intro:17-1} (or Theorem \ref{Yuanyuantheorem2.2}), we can use the similar  way to that used in
   \cite{Wang-Xu} (see also  \cite{Lu}) to get an algorithm for the optimal time and the optimal control.
\end{itemize}
\end{Remark}

\section{Preliminaries}
\subsection{Exact synchronization}

The main result of this subsection is the next theorem.

\begin{Theorem}\label{Preli:2}
Let $T>0$. The following statements are true:

\noindent $(i)$ If the system (\ref{Intro:1}) is exactly synchronizable at $T$, then
either $(H_1)$ or $(H_2)$ holds.

\noindent $(ii)$ If either $(H_1)$ or (\ref{wyuanyuan1.8}) in $(H_2)$ holds, then
the system (\ref{Intro:1}) is exactly synchronizable at $T$.
\end{Theorem}
\begin{Remark}\label{yuanyuanremark2.2}
Theorem \ref{Preli:2} implies that the system (\ref{Intro:1}) is exactly synchronizable at each time $T$ if and only if either $(H_1)$ or $(H_2)$ is true.
\end{Remark}

The proof of  Theorem \ref{Preli:2} needs the next  null controllability result quoted from \cite{Ammar}.
 \begin{Lemma}\label{Preli:1}
Let $T>0$, $O\in \mathbb R^{k\times k}$ and $L\in\mathbb R^{k\times l}$, where $k$, $l\geq 1$.
Then the system
\begin{equation*}
\left\{
\begin{array}{lll}
\bm w_t-\Delta \bm w+O\bm w=\chi_\omega L\bm u&\mbox{in}&\Omega\times (0,+\infty),\\
\bm w=\bm 0&\mbox{on}&\partial \Omega\times(0,+\infty),\\
\bm w(0)\in L^2(\Omega)^k&&
\end{array}
\right.
\end{equation*}
is null controllable at $T$ if and only if
\begin{equation*}
\mbox{rank}(L, OL, \dots, O^{k-1}L)=k.
\end{equation*}
\end{Lemma}
The proof of  Theorem~\ref{Preli:2} also needs another lemma. To state it, we introduce the following controlled system:
\begin{equation}\label{wang1.6-1}
\left\{
\begin{array}{lll}
\bm z_t-\Delta \bm z+\widetilde{A}\bm z=\chi_\omega DB\bm u&\mbox{in}&\Omega\times (0,+\infty),\\
\bm z=\bm 0&\mbox{on}&\partial\Omega\times (0,+\infty),\\
\bm z(0)=\bm z_0&\mbox{in}&\Omega,
\end{array}\right.
\end{equation}
where $\bm u\in L^2(0,+\infty;L^2(\Omega)^m)$ and $\bm z_0\in L^2(\Omega)^{n-1}$.
We write $\bm z(\cdot;\bm z_0,\bm u)$ for the solution of (\ref{wang1.6-1}).

\begin{Lemma}\label{Preli:1-add}
Let $T\geq 0$ and let $\bm y_0\in L^2(\Omega)^n$. The following two conclusions are ture:
\begin{itemize}
\item[$(i)$] Assume that (\ref{wang1.5}) holds, i.e.,
$\sum_{\ell=1}^n a_{i\ell}=\sum_{\ell=1}^n a_{j\ell}$ for all $i, j\in \{1,2,\dots,n\}$.
Then $D\bm y(t;\bm y_0,\bm u)=\bm z(t;D\bm y_0,\bm u)$ for all $t\geq 0$ and $\bm u\in L^2(0,+\infty;L^2(\Omega)^m)$.
\item[$(ii)$] Assume that (\ref{wyuanyuan1.7}) holds, i.e.,
$\sum_{\ell=1}^n a_{i_0\ell}\not=\sum_{\ell=1}^n a_{j_0\ell}$
for some $i_0, j_0\in \{1,2,\dots,n\}$. If there exists a control $\bm u\in L^2(0,+\infty;L^2(\Omega)^m)$
with $\bm u(t)=\bm 0$ for a.e. $t>T$ so that $D\bm y(t;\bm y_0,\bm u)=\bm 0$ for all $t\geq T$,
then $\bm y(t;\bm y_0,\bm u)=\bm 0$ for all $t\geq T$.
\end{itemize}
\end{Lemma}
\begin{proof} (i) The desired result follows from (\ref{wang1.6}), (\ref{Intro:1})
and (\ref{wang1.6-1}) directly.

(ii) Because $D\bm y(t;\bm y_0,\bm u)=\bm 0$ for all $t\geq T$, i.e.,
\begin{equation}\label{Preli:1-add(1)}
y_1(t;\bm y_0,\bm  u)=y_2(t;\bm y_0,\bm  u)
=\cdots=y_n(t;\bm y_0,\bm u)\;\;\mbox{for all}\;\;t\geq T,
\end{equation}
it follows by (\ref{Intro:1}) that
\begin{equation*}
\sum_{\ell=1}^n( a_{i_0 \ell}- a_{j_0 \ell}) y_{i_0}(t;\bm y_0,\bm u)=0\;\;\mbox{for all}\;\;t\geq T.
\end{equation*}
Since $\sum_{\ell=1}^n( a_{i_0 \ell}- a_{j_0 \ell})\neq 0$, the above, along with (\ref{Preli:1-add(1)}), yields the desired result.

\vskip 5pt
Hence, we finish the proof of Lemma \ref{Preli:1-add}.
\end{proof}

\noindent\textbf{\it{Proof of Theorem~\ref{Preli:2}.}}
Arbitrarily fix $T>0$.
 We first show $(i)$.  Assume that (1.1) is exactly synchronizable at $T$.
About $\{a_{ij}\}_{i,j=1}^n$, there are only two possibilities:
either $\sum_{\ell=1}^n a_{i\ell}=\sum_{\ell=1}^n a_{j\ell}$ for all $i, j\in \{1,2,\dots,n\}$ or $\sum_{\ell=1}^n a_{i_0 \ell}\neq \sum_{\ell=1}^n a_{j_0 \ell}$
for some $i_0, j_0\in \{1,2,\dots,n\}$.

In the case that the first possibility occurs,
 we arbitrarily fix $\bm z_0\in L^2(\Omega)^{n-1}$.
Denote $\bm z_0\triangleq (z_{0,1},z_{0,2},\dots,z_{0,n-1})^{\top}$.
We can easily check that $\bm z_0=D\bm {\widehat{y}}_0$
for some $\bm {\widehat{y}}_0\in L^2(\Omega)^n$. For example, we can write
\begin{equation*}
\bm {\widehat{y}}_0=\Big(\sum_{k=1}^{n-1} z_{0,k}, \sum_{k=2}^{n-1}z_{0,k},\dots,z_{0,n-1},0\Big)^{\top}.
\end{equation*}
Since (\ref{Intro:1}) is exactly synchronizable at $T$,
there exists a control $\bm {\widehat{u}}\in L^2(0,+\infty;L^2(\Omega)^m)$ with $\bm {\widehat{u}}(t)=\bm 0$ a.e. $t>T$
so that
\begin{equation}\label{Preli:2-4}
D\bm y(t;\bm {\widehat{y}}_0,\bm {\widehat{u}})=\bm 0\;\;\mbox{for all}\;\;t\geq T.
\end{equation}
From (i) of Lemma~\ref{Preli:1-add} and (\ref{Preli:2-4}) it follows that
(\ref{wang1.6-1}) is null controllable at $T$.
This, together with Lemma~\ref{Preli:1} (where $L=DB, O=\widetilde{A}, k=n-1$ and $l=m$), implies that
\begin{equation*}
\textrm{rank}(DB,\widetilde ADB,\dots, \widetilde A^{n-2}DB)=n-1,
\end{equation*}
which, combined with (\ref{wang1.6}), indicates that
\begin{equation*}
\textrm{rank}(DB,DAB,\dots,DA^{n-2}B)=n-1.
\end{equation*}
Hence, $(H_1)$ is true.

We now consider the  case that the second possibility occurs.
Since the system (\ref{Intro:1}) is exactly synchronizable at $T$,
we see that for each $\bm y_0\in L^2(\Omega)^n$, there exists $\widetilde {\bm u}\in L^2(0,+\infty;L^2(\Omega)^m)$ with
$\widetilde {\bm u}(t)=\bm 0$ for a.e. $t>T$ so that
$D\bm y(t;\bm y_0,\widetilde{\bm u})=\bm 0$ for all $t\geq T$.
Then we can apply  (ii) of Lemma~\ref{Preli:1-add}
to this case to obtain  that
$\bm y(t;\bm y_0,\widetilde{\bm u})=\bm 0$ for all $t\geq T$, i.e.,
the system (\ref{Intro:1}) is null controllable at $T$.
This, along with Lemma~\ref{Preli:1} (where $L=B, O=A, k=n$ and $l=m$), implies that
\begin{equation*}
  \mbox{rank}(B,AB,\dots,A^{n-1}B)= n.
\end{equation*}
Thus, $(H_2)$ is true. This ends the proof of the conclusion $(i)$. \\

We next show the conclusion $(ii)$ by the following two steps.

In Step 1, we consider the case  that  $(H_1)$ is true.
By the note $(b_3)$ in section 1, we can find a unique matrix
$\widetilde A\in \mathbb{R}^{n-1}\times\mathbb{R}^{n-1}$ holding  (\ref{wang1.6}).
This, along with (\ref{wang1.5-1}), indicates that
\begin{equation}\label{Preli:2-8}
\textrm{rank}(DB,\widetilde ADB,\dots, \widetilde A^{n-2}DB)=\textrm{rank}(DB,DAB,\dots,DA^{n-2}B)=n-1.
\end{equation}
By (\ref{Preli:2-8}), we can apply  Lemma~\ref{Preli:1} (where $L=DB$, $O=\widetilde A$,
 $k=n-1$ and $l=m$) to see that the system (\ref{wang1.6-1}) is null controllable at $T$. Thus,
for $\bm z_0=D \bm y_0$, with $\bm y_0\in L^2(\Omega)^n$ arbitrarily fixed,
there exists a control ${\bm u_0}\in L^2(0,+\infty;L^2(\Omega)^m)$, with
${\bm u_0}(t)=\bm 0$ for a.e. $t>T$, so that
\begin{equation}\label{yuanyuan2.4}
\bm z(t;D\bm y_0,\bm u_0)=\bm 0\;\;\mbox{for all}\;\;t\geq T,
\end{equation}
where $z(\cdot; D\bm y_0,\bm u_0)$ is the solution of  (\ref{wang1.6-1}) with $\bm z_0=D \bm y_0$ and $\bm u=\bm u_0$.
Because (\ref{wang1.5}) is a part of $(H_1)$, we can use (i) of Lemma~\ref{Preli:1-add},
as well as (\ref{yuanyuan2.4}), to obtain
 that
\begin{equation*}
D\bm y(t;\bm y_0,{\bm u_0})=\bm 0\;\;\mbox{for all}\;\;t\geq T,
\end{equation*}
from which, we see that the system (\ref{Intro:1}) is exactly synchronizable at $T$.

In Step 2, we consider the case that (\ref{wyuanyuan1.8}) in $(H_2)$ is true.
 In this case,  we can use  Lemma~\ref{Preli:1}
 (where $O=A$, $L=B$, $k=n$ and $l=m$) to find that  the system (\ref{Intro:1})
is null controllable at $T$. Consequently, it is exactly synchronizable at $T$.

\vskip 5pt

In summary, we finish the proof of Theorem~\ref{Preli:2}.
\hspace{60mm} $\Box$\\

\subsection{Transformations of problems}
The hypothesis $(H_1)$ contains two parts: the row condition (\ref{wang1.5})
and the rank condition (\ref{wang1.5-1}), while the hypothesis $(H_2)$ also contains
two parts: the row condition (\ref{wyuanyuan1.7}) and the rank condition (\ref{wyuanyuan1.8}).
Each row condition can help  us to transform the problem
$(TP)_M^{\bm y_0}$  into a minimal time control problem where the target is the origin of the state space.
The same can be said about $(NP)_T^{\bm y_0}$. More precisely, we have the next discussions.

Under the row condition (\ref{wang1.5}) in $(H_1)$, we consider the following two problems:
\begin{equation*}
\begin{array}{lll}
(\widetilde{TP})_M^{\bm y_0}\;\;\;\;\;\;\;\;\;{\widetilde T}(M,\bm y_0)
\triangleq\mbox{inf}_{\bm u\in \mathcal{U}_M}\{T\geq 0:&
\bm z(T;D\bm y_0,\bm u)=\bm 0\\
&\mbox{and}\;\;\bm u(\cdot)=\bm 0\;\;\mbox{over}\;\;[T,+\infty)\},
\end{array}
\end{equation*}
and
\begin{equation*}
\begin{array}{lll}
(\widetilde{NP})_T^{\bm y_0}\;\;\;\;\;\;\;\;\;{\widetilde N}(T,\bm y_0)
\triangleq\mbox{inf} \{\|\bm v\|_{L^2(0,+\infty;L^2(\Omega)^m)}:&\bm z(T;D\bm y_0,\bm v)=\bm 0\\
&\mbox{and}\;\bm v(\cdot)=\bm 0\;\,\mbox{over}\;\,[T,+\infty)\}.
\end{array}
\end{equation*}

About Problem $(\widetilde{TP})_M^{\bm y_0}$, we call ${\widetilde T}(M,\bm y_0)$ the optimal time; we call $\bm u\in \mathcal{U}_M$
    an admissible control if there is $T\geq 0$ so that $\bm z(T;D\bm y_0,\bm u)=\bm 0$ and $\bm u(\cdot)=\bm 0$
     over $[T,+\infty)$; we call $\bm u^*\in \mathcal{U}_M$
    an optimal control if ${\widetilde T}(M,\bm y_0)<+\infty, \bm z({\widetilde T}(M,\bm y_0);D\bm y_0,\bm u^*)=\bm 0$
     and $\bm u^*(\cdot)=\bm 0$
 over $[{\widetilde T}(M,\bm y_0),+\infty)$; we agree that ${\widetilde T}(M,\bm y_0)=+\infty$ if the problem
 $(\widetilde{TP})_M^{\bm y_0}$ has no admissible control.

About Problem $(\widetilde{NP})_T^{\bm y_0}$,
we call $\bm v\in L^2(0,+\infty;L^2(\Omega)^m)$ an admissible control if $\bm z(T;D\bm y_0,\bm v)=\bm 0$ and
$\bm v(\cdot)=\bm 0$
  over $[T, +\infty)$; we call $\widetilde{N}(T,\bm y_0)$  the minimal norm; we call a function $\bm v^*$
 an optimal control if it is admissible and satisfies that $\|\bm v^*\|_{L^2(0,+\infty;L^2(\Omega)^m)}=\widetilde{N}(T,\bm y_0)$.

Then we have the next theorem.
\begin{Theorem}\label{20180225-1}
Suppose that  the row condition (\ref{wang1.5}) in $(H_1)$ holds.
Let $M>0, T>0$ and $\bm y_0\in L^2(\Omega)^n$.
Then the following conclusions are true:
\begin{itemize}
\item[$(i)$] Problems $(TP)_M^{\bm y_0}$ and $(\widetilde{TP})_M^{\bm y_0}$ have the same optimal time,
the same admissible controls and the same optimal controls.
\item[$(ii)$] Problems $(NP)_T^{\bm y_0}$ and $(\widetilde{NP})_T^{\bm y_0}$ have the same minimal norm,
the same admissible controls and the same optimal controls.
\end{itemize}
\end{Theorem}
\begin{proof}
By the definition of
$S$ and (\ref{wang1.3}), one can easily check that  $\bm y\in S$ is equivalent to $D\bm y=\bm 0$. Then by the row condition (\ref{wang1.5}) in $(H_1)$, we can apply
  (i) of Lemma~\ref{Preli:1-add} to obtain the desired conclusions. This completes the proof.
\end{proof}

Under the row condition (\ref{wyuanyuan1.7}) in $(H_2)$, we consider the following two problems:
\begin{equation*}
\begin{array}{lll}
(\widehat{TP})_M^{\bm y_0}\;\;\;\;\;\;\;\;\;{\widehat T}(M,\bm y_0)
\triangleq\mbox{inf}_{\bm u\in \mathcal{U}_M}\{T\geq 0:&
\bm y(T;\bm y_0,\bm u)=\bm 0\\
&\mbox{and}\;\;\bm u(\cdot)=\bm 0\;\;\mbox{over}\;\;[T,+\infty)\},
\end{array}
\end{equation*}
and
\begin{equation*}
\begin{array}{lll}
(\widehat{NP})_T^{\bm y_0}\;\;\;\;\;\;\;\;\;{\widehat N}(T,\bm y_0)
\triangleq\mbox{inf} \{\|\bm v\|_{L^2(0,+\infty;L^2(\Omega)^m)}:&\bm y(T;\bm y_0,\bm v)=\bm 0\\
&\mbox{and}\;\bm v(\cdot)=\bm 0\;\,\mbox{over}\;\,[T,+\infty)\}.
\end{array}
\end{equation*}

About Problem $(\widehat{TP})_M^{\bm y_0}$, we call ${\widehat T}(M,\bm y_0)$ the optimal time; we call $\bm u\in \mathcal{U}_M$
    an admissible control if there is $T\geq 0$ so that $\bm y(T;\bm y_0,\bm u)=\bm 0$ and $\bm u(\cdot)=\bm 0$
     over $[T,+\infty)$; we call $\bm u^*\in \mathcal{U}_M$
    an optimal control if ${\widehat T}(M,\bm y_0)<+\infty, \bm y({\widehat T}(M,\bm y_0);\bm y_0,\bm u^*)=\bm 0$
     and $\bm u^*(\cdot)=\bm 0$
 over $[{\widehat T}(M,\bm y_0),+\infty)$; we agree that ${\widehat T}(M,\bm y_0)=+\infty$ if the problem
 $(\widehat{TP})_M^{\bm y_0}$ has no admissible control.

About Problem $(\widehat{NP})_T^{\bm y_0}$, we call $\bm v\in L^2(0,+\infty;L^2(\Omega)^m)$ an admissible control
if $\bm y(T;\bm y_0,\bm v)=\bm 0$ and $\bm v(\cdot)=\bm 0$
  over $[T, +\infty)$; we call $\widehat{N}(T,\bm y_0)$  the minimal norm;
 we call a function $\bm v^*$
 an optimal control if it is admissible and satisfies that $\|\bm v^*\|_{L^2(0,+\infty;L^2(\Omega)^m)}=\widehat{N}(T,\bm y_0)$.

Then we have the next theorem.
\begin{Theorem}\label{20180225-2}
Suppose that the row condition (\ref{wyuanyuan1.7}) in $(H_2)$ holds.
Let $M>0, T>0$ and $\bm y_0\in L^2(\Omega)^n$. Then
the following conclusions are true:
\begin{itemize}
\item[$(i)$] Problems $(TP)_M^{\bm y_0}$ and $(\widehat{TP})_M^{\bm y_0}$ have the same optimal time,
the same admissible controls and the same optimal controls.
\item[$(ii)$] Problems $(NP)_T^{\bm y_0}$ and $(\widehat{NP})_T^{\bm y_0}$ have the same minimal norm,
the same admissible controls and the same optimal controls.
\end{itemize}
\end{Theorem}
\begin{proof}
By the row condition (\ref{wyuanyuan1.7}) in $(H_2)$, we can apply
  (ii) of Lemma~\ref{Preli:1-add} to obtain the desired conclusions.
This completes the proof.
\end{proof}

From now on, we will transform studies on $(TP)_M^{\bm y_0}$ (or $(NP)_T^{\bm y_0}$)
into studies on the problem $(\widetilde{TP})_M^{\bm y_0}$ (or
 the problem $(\widetilde{NP})_T^{\bm y_0}$), when $(H_1)$ is assumed;
 we will transform studies on $(TP)_M^{\bm y_0}$ (or $(NP)_T^{\bm y_0}$) into studies on the problem $(\widehat{TP})_M^{\bm y_0}$
 (or the problem $(\widehat{NP})_T^{\bm y_0}$), when $(H_2)$ is assumed.

{\it Throughout the rest of the paper, in the proofs of all results related to $(TP)_M^{\bm y_0}$ and $(NP)_T^{\bm y_0}$, we only deal with the case where $(H_1)$ is assumed.} The case where
$(H_2)$ is assumed can be studied very similarly.  The reasons are as follows:

\begin{description}
 \item[$(g_1)$] Both $(H_1)$ and $(H_2)$ are constituted by two parts: a row condition and
   a rank condition.
\item[$(g_2)$] The problems $(TP)_M^{\bm y_0}$ and $(NP)_T^{\bm y_0}$ have been transformed into
the problems $(\widetilde{TP})_M^{\bm y_0}$ and $(\widetilde{NP})_T^{\bm y_0}$,
when the row condition (\ref{wang1.5}) in $(H_1)$ holds;
    the problems  $(\widehat{TP})_M^{\bm y_0}$ and $(\widehat{NP})_T^{\bm y_0}$,
    when the row condition (\ref{wyuanyuan1.7}) in $(H_2)$ holds.
\item [$(g_3)$] Throughout the rest of the paper,
in the studies on problems $(\widetilde{TP})_M^{\bm y_0}$ and $(\widetilde{NP})_T^{\bm y_0}$
(or problems $(\widehat{TP})_M^{\bm y_0}$ and $(\widehat{NP})_T^{\bm y_0}$), the row condition (\ref{wang1.5})
(or the row condition (\ref{wyuanyuan1.7})) will not be used again.
\item[$(g_4)$] The rank condition  (\ref{wang1.5-1}) in $(H_1)$ and
the rank condition (\ref{wyuanyuan1.8}) in $(H_2)$ are the same essentially.
This can be explained in the following manner:
    First, the  controlled system
 (\ref{wang1.6-1}) (to the  problems  $(\widetilde{TP})_M^{\bm y_0}$ and $(\widetilde{NP})_T^{\bm y_0}$)
  is governed by the pair of matrices $(\widetilde{A}, DB)$ and is in the state space
  $L^2(\Omega)^{n-1}$, while the controlled system (\ref{Intro:1})
  (to the problems $(\widehat{TP})_M^{\bm y_0}$ and $(\widehat{NP})_T^{\bm y_0}$) is governed by the pair of matrices $(A,B)$
  and is in the state space $L^2(\Omega)^{n}$. Second,
      since $\widetilde{A}$ satisfies (\ref{wang1.6}) (see the note $(b_3)$ in section 1), we see that (\ref{wang1.5-1}) is equivalent to the rank condition:
\begin{equation*}
\mbox{rank}(DB,\widetilde{A}DB,\dots,\widetilde{A}^{n-2}DB)=n-1.
\end{equation*}
Hence, the rank condition  (\ref{wang1.5-1}) (in $(H_1)$) for the system  (\ref{wang1.6-1}) is the same as the rank condition (\ref{wyuanyuan1.8}) (in $(H_2)$) for the system (\ref{Intro:1}) essentially.
\end{description}

\section{Minimal norm control problem}
This section studies properties on the minimal norm control problem $(NP)_T^{\bm y_0}$.
\subsection{Characteristic for minimal norm controls}

The main purpose of this subsection is to prove the next theorem which gives the characteristic of minimal norm controls.

\begin{Theorem}\label{Norm:6} Let $T>0$ and let $\bm y_0\in L^2(\Omega)^n$.
 The following conclusions are true:
\begin{itemize}
  \item[(i)] Suppose that $(H_1)$ holds. If $\bm y_0\in S$, then $\bm 0$ is the unique optimal control to
  the problem $(NP)_T^{\bm y_0}$; If $\bm y_0\not\in S$, then the control, defined by
\begin{equation}\label{Norm:6-1}
\bm v^*(t)\triangleq \left\{
\begin{array}{ll}
\chi_\omega B^\top D^\top \bm\psi^*(t),&t\in (0,T),\\
\bm 0,&t\geq T,
\end{array}\right.
\end{equation}
is the unique optimal control to $(NP)_{T}^{\bm y_0}$, where
$\chi_\omega B^\top D^\top \bm\psi^*$, with $\bm \psi^*\in C([0,T);L^2(\Omega)^{n-1})$ solving
(\ref{Intro:11}), is the unique minimizer of
$J_{T,1}^{\bm y_0}$ over $Y_{T,1}$, moreover,  $\bm v^*(t)\not=\bm 0$ for a.e. $t\in (0,T)$.

  \item[(ii)] Suppose that $(H_2)$ holds. If $\bm y_0=\bm 0$,
  then $\bm 0$ is the unique optimal control to
  the problem $(NP)_T^{\bm y_0}$; If $\bm y_0\not=\bm 0$, then the control, defined by
\begin{equation*}
\bm v^*(t)\triangleq \left\{
\begin{array}{ll}
\chi_\omega B^\top \bm\varphi^*(t),&t\in (0,T),\\
\bm 0,&t\geq T,
\end{array}\right.
\end{equation*}
is the unique optimal control to $(NP)_{T}^{\bm y_0}$, where
$\chi_\omega B^\top \bm\varphi^*$, with $\bm \varphi^*\in C([0,T);L^2(\Omega)^n)$ solving (\ref{Intro:14}),
is the unique minimizer of
$J_{T,2}^{\bm y_0}$ over $Y_{T,2}$, moreover,  and $\bm v^*(t)\not=\bm 0$ for a.e. $t\in (0,T)$.
\end{itemize}
\end{Theorem}
To show Theorem \ref{Norm:6}, we need some preparations.

\vskip 5pt

\subsubsection{Observability estimates}

\begin{Proposition}\label{Preli:5} Let $T>0$. The following conclusions are true:
\begin{itemize}
\item[(i)] Suppose that $(H_1)$ holds.
Then there exists a positive constant $C(T)$ so that
\begin{equation}\label{Preli:5-1}
\|\bm \psi(0;T,\bm \psi_T)\|_{L^2(\Omega)^{n-1}}^2\leq
C(T)\displaystyle{\int_0^T}\|\chi_\omega B^\top D^\top\bm \psi(t;T,\bm\psi_T)\|_{L^2(\Omega)^m}^2\mathrm dt
\end{equation}
for all $\bm\psi_T\in L^2(\Omega)^{n-1}$.
Here, $\bm\psi(\cdot;T,\bm\psi_T)$ denotes the unique solution to
(\ref{Intro:11}) with the initial condition $\bm \psi(T)=\bm \psi_T$.
\item[(ii)] Suppose that $(H_2)$ holds.
Then there exists a positive constant $C(T)$ so that
\begin{equation*}
\|\bm \varphi(0;T,\bm \varphi_T)\|_{L^2(\Omega)^n}^2\leq
 C(T)\displaystyle{\int_0^T} \|\chi_\omega B^\top \bm \varphi(t;T,\bm\varphi_T)\|_{L^2(\Omega)^m}^2\mathrm dt
\end{equation*}
for all $\bm\varphi_T\in L^2(\Omega)^n$. Here, $\bm\varphi(\cdot;T,\bm\varphi_T)$ denotes the unique solution to
(\ref{Intro:14}) with the initial condition $\bm \varphi(T)=\bm \varphi_T$.
\end{itemize}
\end{Proposition}

To prove it, we need the next observability estimate quoted from \cite{Ammar}.

\begin{Lemma}\label{Preli:4}
Let $O\in \mathbb R^{k\times k}$, $L\in\mathbb R^{k\times l}$ and $T>0$, where $k$, $l\geq 1$. Assume that
\begin{equation*}
\mbox{rank}(L, OL,\dots, O^{k-1}L)=k.
\end{equation*}
Then any solution $\bm \varphi(\cdot)$ to the following system:
\begin{equation*}
\left\{
\begin{array}{lll}
\bm \varphi_t+\Delta \bm \varphi-O^{\top}\bm \varphi=\bm 0&\mbox{in}&\Omega\times (0,T),\\
\bm \varphi=\bm 0&\mbox{on}&\partial \Omega\times(0,T),\\
\bm \varphi(T)\in L^2(\Omega)^k&&
\end{array}
\right.
\end{equation*}
satisfies the estimate
\begin{equation*}
\|\bm \varphi(0)\|_{L^2(\Omega)^k}^2\leq
C(T)\int_0^{T}\|\chi_\omega L^{\top}\bm\varphi(t)\|^2_{L^2(\Omega)^l}\mathrm dt,
\end{equation*}
where  $C(T)$ is a positive constant depending on $T$.
\end{Lemma}

\noindent\textbf{\it{Proof of Proposition~\ref{Preli:5}.}}
We only need to  prove the conclusion  (i), because of the reasons given in the last paragraph in subsection 3.2.
 Suppose  that
$(H_1)$ holds. Then we have that (see (\ref{Preli:2-8}))
\begin{equation*}
\mbox{rank}(DB,\widetilde{A}DB,\dots,\widetilde{A}^{n-2}DB)=n-1.
\end{equation*}
This, together with Lemma~\ref{Preli:4} (where $O=\widetilde{A}, L=DB, k=n-1$ and $l=m$), implies (\ref{Preli:5-1}).
\hspace{145mm}$\Box$

\subsubsection{Properties of functionals}

Recall (\ref{Intro:12}) and (\ref{Intro:15}). We have descriptions on $Y_{T,1}$ and $Y_{T,2}$ as follows.

\begin{Lemma}\label{Norm:2}Let $T>0$. The following conclusions are true:
\begin{itemize}
\item[(i)] Suppose that $(H_1)$ holds. Then
\begin{equation}\label{20180205-2}
\begin{array}{l}
Y_{T,1}=\{\chi_\omega B^\top D^\top \bm \psi\in L^2(0,T;L^2(\omega)^m):\;
\bm \psi\in C([0,T);L^2(\Omega)^{n-1})\;\;\mbox{solves}\\
\;\;\;\;\;\;\;\;\;\;\;\mbox{(\ref{Intro:11}) and}\;\;
\chi_\omega B^\top D^\top \bm \psi(\cdot;T,\bm z_i)\rightarrow \chi_\omega B^\top D^\top \bm \psi(\cdot)
\;\;\mbox{strongly in}\\
\;\;\;\;\;\;\;\;\;\;\;\;\;L^2(0,T;L^2(\omega)^m)\;\;
\mbox{for some sequence}\;\;\{\bm z_i\}_{i\geq 1}\subseteq L^2(\Omega)^{n-1}\}.
\end{array}
\end{equation}
Moreover, there exists a positive constant $C(T)$ so that
\begin{equation}\label{Norm:2-2}
\|\bm \psi(0)\|_{L^2(\Omega)^{n-1}}^2\leq C(T)\int_0^T \|\chi_\omega B^\top D^\top \bm \psi\|_{L^2(\omega)^m}^2\mathrm dt
\;\,\mbox{for all}\;\;\chi_\omega B^\top D^\top \bm \psi\in Y_{T,1}.
\end{equation}
\item[(ii)] Suppose that  $(H_2)$ holds. Then
\begin{equation*}
\begin{array}{l}
Y_{T,2}=\{\chi_\omega B^\top  \bm \varphi\in L^2(0,T;L^2(\omega)^m):\;
\bm \varphi\in C([0,T);L^2(\Omega)^n)\;\;\mbox{solves}\\
\;\;\;\;\;\;\;\;\;\;\;\mbox{(\ref{Intro:14}) and}\;\;
\chi_\omega B^\top\bm \varphi(\cdot;T,\bm z_i)\rightarrow \chi_\omega B^\top \bm \varphi(\cdot)\;\;\mbox{strongly in}\\
\;\;\;\;\;\;\;\;\;\;\;\;\;L^2(0,T;L^2(\omega)^m)\;\;\mbox{for some sequence}\;\;
\{\bm z_i\}_{i\geq 1}\subseteq L^2(\Omega)^n\}.
\end{array}
\end{equation*}
Moreover, there exists a positive constant $C(T)$ so that
\begin{equation*}
\|\bm \varphi(0)\|_{L^2(\Omega)^n}^2\leq C(T)\int_0^T \|\chi_\omega B^\top \bm \varphi\|_{L^2(\omega)^m}^2\mathrm dt
\;\;\mbox{for all}\;\;\chi_\omega B^\top \bm \varphi\in Y_{T,2}.
\end{equation*}
\end{itemize}
\end{Lemma}
The proof of Lemma~\ref{Norm:2} is similar to that of Lemma 2.1 in \cite{Wang-Xu-Zhang}. For the sake of
completeness, we give its detailed proof below.

\vskip 5pt

\noindent\textbf{\it{Proof of Lemma~\ref{Norm:2}.}}
We only need to  prove the conclusion  (i), because of the reasons given in the last paragraph in subsection 3.2.
Let $\bm \xi\in Y_{T,1}$. According to (\ref{Intro:12}), there exists a
sequence $\{\bm z_i\}_{i\geq 1}\subseteq L^2(\Omega)^{n-1}$ so that
\begin{equation}\label{Norm:2-3}
\chi_\omega B^\top D^\top \bm \psi(\cdot;T,\bm z_i)\rightarrow \bm \xi\;\;
\mbox{strongly in}\;\;L^2(0,T;L^2(\omega)^m).
\end{equation}
Here $\bm \psi(\cdot;T,\bm z_i)$ denotes the unique solution to (\ref{Intro:11}) with the initial condition
$\bm \psi(T)=\bm z_i$. From (\ref{Norm:2-3}) it follows that $\{\chi_\omega B^\top D^\top \bm \psi(\cdot;T,\bm z_i)\}_{i\geq 1}$
is bounded in $L^2(0,T;L^2(\omega)^m)$. Let $\{T_\ell\}_{\ell\geq 1}\subseteq (0,T)$ be such that
$T_\ell\nearrow T$. Arbitrarily fix  $\ell\geq 1$. By (\ref{Preli:5-1}) in Proposition~\ref{Preli:5}, we have that for all $i\geq 1$,
\begin{equation}\label{Norm:2-4}
\|\bm \psi(T_{\ell+1};T,\bm z_i)\|_{L^2(\Omega)^{n-1}}^2
\leq C(T,T_{\ell+1})\int_{T_{\ell+1}}^T \|\chi_\omega B^\top D^\top \bm \psi(t;T,\bm z_i)\|_{L^2(\omega)^m}^2\mathrm dt
\leq C(\ell).
\end{equation}
We arbitrarily take two subsequences $\{\bm z_{1,i}\}_{i\geq 1}$
and $\{\bm z_{2,i}\}_{i\geq 1}$ from $\{\bm z_{i}\}_{i\geq 1}$. Then,
by  (\ref{Norm:2-4}) and by using the standard $L^2$-theory for parabolic system and Arzel\`{a}-Ascoli theorem to
 $\{\bm \psi(\cdot;T,\bm z_{1,i})\}_{i\geq 1}$ and $\{\bm \psi(\cdot;T,\bm z_{2,i})\}_{\geq 1}$, we can easily see that
there are respectively two subsequences of $\{\bm z_{1,i}\}_{i\geq 1}$
and $\{\bm z_{2,i}\}_{i\geq 1}$, still denoted by themselves,
so that
\begin{equation*}
\bm \psi(\cdot;T,\bm z_{1,i})\rightarrow \bm \psi_{\ell,1}\;\;
\mbox{and}\;\;\bm \psi(\cdot;T,\bm z_{2,i})\rightarrow \bm \psi_{\ell,2}
\;\;\mbox{strongly in}\;\;C([0,T_\ell];L^2(\Omega)^{n-1}),
\end{equation*}
where $\bm \psi_{\ell,1}(\cdot)$ and $\bm \psi_{\ell,2}(\cdot)$ are solutions to (\ref{Intro:11})
(where $T$ is replaced by $T_\ell$). These, together with (\ref{Norm:2-3}), imply that
\begin{equation*}
\chi_\omega B^\top D^\top \bm \psi_{\ell,1}(t)=\chi_\omega B^\top D^\top \bm \psi_{\ell,2}(t)=\bm \xi(t)
\;\;\mbox{for a.e.}\;\;t\in (0,T_\ell).
\end{equation*}
Hence, by (\ref{Preli:5-1}) in Proposition~\ref{Preli:5}, we get that
\begin{equation*}
\bm \psi_\ell(t)\triangleq \bm \psi_{\ell,1}(t)=\bm \psi_{\ell,2}(t)\;\;\mbox{for each}\;\;t\in [0,T_\ell].
\end{equation*}
Then
\begin{equation}\label{Norm:2-5}
\bm \psi(\cdot;T,\bm z_i)\rightarrow \bm \psi_\ell(\cdot)\;\;\mbox{strongly in}\;\;
C([0,T_\ell];L^2(\Omega)^{n-1})
\end{equation}
and
\begin{equation}\label{Norm:2-5(1)}
\bm \xi(t)=\chi_\omega B^\top D^\top \bm \psi_\ell(t)\;\;\mbox{for a.e.}\;\;t\in (0,T_\ell).
\end{equation}
Since $\ell\geq 1$ was arbitrarily taken, it follows from (\ref{Norm:2-5}) that
\begin{equation}\label{Norm:2-6}
\bm \psi_\ell(t)=\bm \psi_{\ell+\ell^\prime}(t)\;\;\mbox{for each}\;\;t\in [0,T_\ell]\;\;
\mbox{and}\;\;\ell^\prime\geq 1.
\end{equation}

We now define
\begin{equation}\label{Norm:2-7}
\bm \psi(t)\triangleq \bm \psi_\ell(t)\;\;\mbox{for all}\;\;t\in [0,T_\ell].
\end{equation}
Then,
according to (\ref{Norm:2-5(1)})-(\ref{Norm:2-7}), $\bm \psi(\cdot)$ is well defined,
$\bm \psi(\cdot)\in C([0,T);L^2(\Omega)^{n-1})$ solves (\ref{Intro:11}) and
$\bm \xi(t)=\chi_\omega B^\top D^\top \bm \psi(t)$ for a.e. $t\in (0,T)$.
The rest is to show that the above $\bm \psi$ satisfies (\ref{Norm:2-2}).
To this end, we use  (\ref{Norm:2-3}), (\ref{Norm:2-5}) and (\ref{Norm:2-7}) to find that
\begin{equation}\label{Norm:2-8}
\chi_\omega B^\top D^\top \bm \psi(\cdot;T,\bm z_i)\rightarrow \chi_\omega B^\top D^\top \bm \psi(\cdot)
\;\;\mbox{strongly in}\;\;L^2(0,T;L^2(\omega)^m)
\end{equation}
and
\begin{equation}\label{Norm:2-9}
\bm \psi(0;T,\bm z_i)\rightarrow \bm \psi(0)\;\;\mbox{strongly in}\;\;L^2(\Omega).
\end{equation}
Meanwhile, we use  (\ref{Preli:5-1}) of Proposition~\ref{Preli:5} to obtain that
\begin{equation*}
\|\bm \psi(0;T,\bm z_i)\|_{L^2(\Omega)^{n-1}}^2\leq
C(T)\int_0^T \|\chi_\omega B^\top D^\top \bm \psi(t;T,\bm z_i)\|_{L^2(\omega)^m}^2\mathrm dt.
\end{equation*}
Finally, by passing to the limit for $i\rightarrow +\infty$ in the latter inequality,
using  (\ref{Norm:2-8}) and (\ref{Norm:2-9}), we obtain (\ref{Norm:2-2}).

Thus, we finish the proof of Lemma \ref{Norm:2}.
\hspace{76mm}$\Box$

\vskip 5pt

Recall (\ref{Intro:13}) and (\ref{Intro:16}).
Based on Lemma~\ref{Norm:2}, we have the following corollary.
\begin{Corollary}\label{20180205-1}
Let $T>0$ and let $\bm y_0\in L^2(\Omega)^n$. The following conclusions are true:
\begin{itemize}
  \item[$(i)$] Suppose that $(H_1)$ holds. Then $J_{T,1}^{\bm y_0}$
  is well defined.
  \item[$(ii)$] Suppose that $(H_2)$ holds. Then $J_{T,2}^{\bm y_0}$
  is well defined.
\end{itemize}
\end{Corollary}
\begin{proof} We only show the conclusion  (i).  Let $\bm \psi_1, \bm\psi_2\in
C([0,T); L^2(\Omega)^{n-1})$ satisfy that  $\chi_\omega B^\top D^\top \bm \psi_1, \chi_\omega B^\top D^\top \bm \psi_2\in Y_{T,1}$ and  $\chi_\omega B^\top D^\top \bm \psi_1=\chi_\omega B^\top D^\top \bm \psi_2$.
Then by (i) of Lemma~\ref{Norm:2},
we see that $\bm \psi_1(0)=\bm \psi_2(0)$ in $L^2(\Omega)^{n-1}$.
Thus, we have that
$$
J_{T,1}^{y_0}(\chi_\omega B^\top D^\top \bm \psi_1)= J_{T,1}^{y_0}(\chi_\omega B^\top D^\top \bm \psi_2).
$$
Hence, $J_{T,1}^{y_0}$ is well defined.
This completes the proof.
\end{proof}

The next lemma presents  some properties of the functionals $J_{T,1}^{\bm y_0}$ and $J_{T,2}^{\bm y_0}$.
\begin{Lemma}\label{Norm:5}
Let $T>0$ and let $\bm y_0\in L^2(\Omega)^n$. The following conclusions are true:
\begin{itemize}
\item[(i)] Suppose that $(H_1)$ holds.
If $\bm y_0\in S$, then $J_{T,1}^{\bm y_0}$ has the unique minimizer $\bm 0$ in $Y_{T,1}$;
If $\bm y_0\not\in S$, then $J_{T,1}^{\bm y_0}$
 has a unique minimizer (in $Y_{T,1}$) which is nontrivial.
\item[(ii)] Suppose that $(H_2)$ holds.
If $\bm y_0=\bm 0$, then $J_{T,2}^{\bm y_0}$ has the unique minimizer $\bm 0$ in $Y_{T,2}$;
If $\bm y_0\not=\bm 0$, then $J_{T,2}^{\bm y_0}$
 has a unique minimizer (in $Y_{T,2}$) which is nontrivial.
\end{itemize}
\end{Lemma}
\begin{proof} We only need to  prove the conclusion  (i), because of the reasons given in the last paragraph in subsection 3.2.

First, we suppose that   $\bm y_0\in S$. Then by the definition of $S$ and
  (\ref{wang1.3}), we have that  $D\bm y_0=\bm 0$.
This, along with
(\ref{Intro:13}) and (\ref{20180205-2}), yields that  $\bm 0$ is the unique minimizer of
$J_{T,1}^{\bm y_0}$ in $Y_{T,1}$.

Next, we suppose that  $\bm y_0\not\in S$. Then by the definition of $S$ and
  (\ref{wang1.3}), we have that   $D\bm y_0\not=\bm 0$.
Since $L^2(0,T;L^2(\omega)^m)$
is reflexive, $Y_{T,1}$, as a closed subspace of $L^2(0,T;L^2(\omega)^m)$ (see (\ref{Intro:12})),
is also reflexive. Meanwhile, by (\ref{Intro:13}) and (i) of Lemma~\ref{Norm:2}, we see that
\begin{equation*}
J_{T,1}^{\bm y_0}: Y_{T,1}\rightarrow \mathbb{R}\;\;\mbox{is continuous, strictly convex and coercive in}\;\;Y_{T,1}.
\end{equation*}
Hence, $J_{T,1}^{\bm y_0}: Y_{T,1}\rightarrow \mathbb{R}$
has a unique minimizer. By (\ref{20180205-2}), this minimizer can be expressed
by $\chi_\omega B^\top D^\top \bm \psi^*\in L^2(0,T;L^2(\omega)^m)$ for some $\bm \psi^*\in C([0,T);L^2(\Omega)^{n-1})$
solving (\ref{Intro:11}).

To end the proof of this lemma, we now only need to prove
that  $\chi_\omega B^\top D^\top \bm \psi^*(\cdot)\not=\bm 0$.
By contradiction,
we suppose that $\chi_\omega B^\top D^\top \bm \psi^*(\cdot)=\bm 0$.
Then we would have that
\begin{equation*}
J_{T,1}^{\bm y_0}(\chi_\omega B^\top D^\top (\alpha \bm \psi))\geq 0\;\;
\mbox{for all}\;\;\alpha\in \mathbb{R}\;\;\mbox{and}\;\;\chi_\omega B^\top D^\top \bm \psi\in X_{T,1}.
\end{equation*}
(Here we used the fact that  $X_{T,1}$ is a subspace of $Y_{T,1}$, see (\ref{Intro:12}).)
From the latter and (\ref{Intro:13}) it follows that
\begin{equation}\label{Norm:5-2}
\langle \bm \psi(0), D\bm y_0\rangle_{L^2(\Omega)^{n-1}}=0\;\;\mbox{for each}
\;\;\chi_\omega B^\top D^\top \bm \psi\in X_{T,1}.
\end{equation}
Let $\bm z$ be the solution to the system:
\begin{equation}\label{Norm:5-3}
\left\{
\begin{array}{lll}
\bm z_t-\Delta \bm z+\widetilde A \bm z=\bm 0&\mbox{in}&\Omega\times (0,T),\\
\bm z=\bm 0&\mbox{on}&\partial\Omega\times (0,T),\\
\bm z(0)=D \bm y_0&\mbox{in}&\Omega.
\end{array}\right.
\end{equation}
Then let $\widetilde{\bm \psi}$ be the solution to the system:
\begin{equation}\label{Norm:5-4}
\left\{
\begin{array}{lll}
\widetilde{\bm \psi}_t+\Delta \widetilde{\bm \psi}-\widetilde A^\top \widetilde{\bm \psi}=\bm 0&\mbox{in}&\Omega\times (0,T),\\
\widetilde{\bm \psi}=\bm 0&\mbox{on}&\partial\Omega\times (0,T),\\
\widetilde{\bm \psi}(T)=\bm z(T)&\mbox{in}&\Omega.
\end{array}\right.
\end{equation}
Since $\bm y_0\in L^2(\Omega)^n$, one can easily use (\ref{Intro:12})
to obtain that  $\chi_\omega B^\top D^\top \widetilde{\bm \psi}\in X_{T,1}$.
Multiplying (\ref{Norm:5-3}) by $\widetilde{\bm \psi}$, then integrating it over $\Omega\times (0,T)$, and then using (\ref{Norm:5-4}),
we obtain that
\begin{equation}\label{Norm:5-5}
\|\bm z(T)\|_{L^2(\Omega)^{n-1}}^2=\langle \widetilde{\bm \psi}(0),D\bm y_0\rangle_{L^2(\Omega)^{n-1}}.
\end{equation}
It follows from (\ref{Norm:5-5}) and (\ref{Norm:5-2}) that $\bm z(T)=\bm 0$. This,
along with (\ref{Norm:5-3}) and the backward uniqueness of parabolic system
(see, for instance,  \cite{Bardos-Tartar}), implies that $D \bm y_0=\bm 0$,
which leads to a contradiction.

\vskip 5pt

In summary, we finish the proof of Lemma \ref{Norm:5}.
\end{proof}

\subsubsection{Proof of Theorem \ref{Norm:6}}
We only need to  prove the conclusion  (i), because of the reasons given in the last paragraph in subsection 3.2.
Notice that we have transformed the problem $(NP)_T^{\bm y_0}$
into the problem $(\widetilde{NP})_T^{\bm y_0}$ under the assumption $(H_1)$ (see (ii) of Theorem~\ref{20180225-1}).

First, we suppose that  $\bm y_0\in S$. Then we have that $D\bm y_0=\bm 0$.
 Hence, $\bm 0$ is the unique optimal control to the problem $(\widetilde{NP})_T^{\bm y_0}$.
 Consequently, $\bm 0$ is the unique optimal control to
$(NP)_T^{\bm y_0}$.

Next, we suppose that  $\bm y_0\not\in S$.
By (i) of Lemma~\ref{Norm:5} and (i) of Lemma \ref{Norm:2}                                                                                                                                                                               , we can let  $\chi_\omega B^\top D^\top \bm\psi^*$
be the unique minimizer of $J_{T,1}^{y_0}$. Let $\bm v^*$ be given by (\ref{Norm:6-1}).
The rest of the proof is organized by  four steps.

\vskip 5pt

Step 1. We show that $\bm v^*(\cdot)$ is admissible for $(NP)_T^{\bm y_0}$.

Since $\chi_\omega B^\top D^\top \bm\psi^*$ is the unique minimizer of
$J_{T,1}^{\bm y_0}$, we have that
for each $\lambda>0$ and $\chi_\omega B^\top D^\top \bm \psi\in Y_{T,1}$,
\begin{equation*}
[J_{T,1}^{\bm y_0}(\chi_\omega B^\top D^\top \bm\psi^*+\lambda\chi_\omega B^\top D^\top \bm\psi)
-J_{T,1}^{\bm y_0}(\chi_\omega B^\top D^\top \bm\psi^*)]/\lambda\geq 0.
\end{equation*}
 Passing to the limit for $\lambda\rightarrow 0^+$ in the above,
using (\ref{Intro:13}), we obtain, after some simple calculations,
 that
for each $\chi_\omega B^\top D^\top \bm \psi\in Y_{T,1}$,
\begin{equation}\label{Norm:6-3}
\int_0^T \langle \chi_\omega B^\top D^\top \bm \psi^*,
\chi_\omega B^\top D^\top \bm \psi\rangle_{L^2(\omega)^m}\mathrm dt
+\langle \bm\psi(0),D\bm y_0\rangle_{L^2(\Omega)^{n-1}}=0.
\end{equation}

Now we  arbitrarily fix $\bm \psi_T\in L^2(\Omega)^{n-1}$.
Let $\bm\psi(\cdot;T,\bm\psi_T)$ be the solution to the system:
\begin{equation}\label{Norm:6-5}
\left\{
\begin{array}{lll}
\bm \psi_t+\Delta \bm \psi-\widetilde A^\top \bm \psi=\bm 0&\mbox{in}&\Omega\times (0,T),\\
\bm \psi=\bm 0&\mbox{on}&\partial\Omega\times (0,T),\\
\bm \psi(T)=\bm \psi_T&\mbox{in}&\Omega.
\end{array}\right.
\end{equation}
Let $\bm z(\cdot;D\bm y_0,\bm v^*)$ be the solution to the system:
\begin{equation*}
\left\{
\begin{array}{lll}
\bm z_t-\Delta \bm z
+\widetilde A \bm z=
\chi_\omega D B\bm v^*&\mbox{in}&\Omega\times (0,+\infty),\\
\bm z=\bm 0&\mbox{on}&\partial\Omega\times (0,+\infty),\\
\bm z(0)=D\bm y_0&\mbox{in}&\Omega.
\end{array}\right.
\end{equation*}
Multiplying (\ref{Norm:6-5}) by $\bm z(\cdot;D\bm y_0,\bm v^*)$, then integrating it over $\Omega\times (0,T)$
and then using the second system, we obtain that
\begin{equation}\label{Norm:6-6}
\begin{array}{lll}
&&\langle \bm \psi_T,\bm z(T;D\bm y_0,\bm v^*)\rangle_{L^2(\Omega)^{n-1}}\\
&=&\langle \bm \psi(0;T,\bm\psi_T),D \bm y_0\rangle_{L^2(\Omega)^{n-1}}+
\displaystyle{\int_0^T} \langle \chi_\omega \bm v^*(t),\chi_\omega B^\top D^\top \bm \psi(t;T,\bm\psi_T)\rangle_{L^2(\omega)^m}\mathrm dt.
\end{array}
\end{equation}
Now it follows from (\ref{Norm:6-1}), (\ref{Norm:6-3}) and (\ref{Norm:6-6}) that
\begin{equation*}
\langle \bm \psi_T,\bm z(T;D\bm y_0,\bm v^*)\rangle_{L^2(\Omega)^{n-1}}=0\;\;\mbox{for all}\;\;
\bm \psi_T\in L^2(\Omega)^{n-1},
\end{equation*}
which indicates that $\bm z(T;D\bm y_0,\bm v^*)=\bm 0$. Thus, $\bm v^*$ is an admissible control to the problem $(\widetilde{NP})_T^{\bm y_0}$. Equivalently, it is an admissible control
to
$(NP)_T^{\bm y_0}$ (see (ii) of Theorem~\ref{20180225-1}).\\

Step 2. We prove that $\bm v^*(\cdot)$ is an optimal control to $(NP)_T^{\bm y_0}$.

We arbitrarily take an    admissible control  $\bm v$
to $(NP)_T^{\bm y_0}$.
Then it is an admissible control to the problem $(\widetilde{NP})_T^{\bm y_0}$ (see (ii) of Theorem~\ref{20180225-1}), which implies that
\begin{equation}\label{Norm:6-7}
\bm z(T;D\bm y_0,\bm v)=\bm 0.
\end{equation}
We aim to show that
\begin{equation}\label{Yuanyuan4.21}
\int_0^T \|\bm v^*\|_{L^2(\omega)^m}^2\mathrm dt
\leq \int_0^T \|\bm v\|_{L^2(\Omega)^m}^2\mathrm dt.
\end{equation}
When it is proved, we achieve the goal.

We now show (\ref{Yuanyuan4.21}). The fact that  $\chi_\omega B^\top D^\top \bm \psi^*\in Y_{T,1}$, along with  (\ref{Intro:12}), yields that  there exists a sequence
$\{\chi_\omega B^\top D^\top \bm \psi_\ell\}_{\ell\geq 1}\subseteq X_{T,1}\subseteq Y_{T,1}$ so that
\begin{equation}\label{Norm:6-7(1)}
\chi_\omega B^\top D^\top \bm \psi_\ell\rightarrow \chi_\omega B^\top D^\top \bm \psi^*\;\;
\mbox{strongly in}\;\;L^2(0,T;L^2(\omega)^m).
\end{equation}
This, along with (\ref{Norm:2-2}) in Lemma~\ref{Norm:2}, implies that
\begin{equation}\label{Norm:6-9}
\bm \psi_\ell(0)\rightarrow \bm \psi^*(0)\;\;\mbox{strongly in}\;\;L^2(\Omega)^{n-1}.
\end{equation}
 For each $\ell\geq 1$, according to (\ref{Intro:12}) and (\ref{Intro:11}),
$\bm \psi_\ell$ is the unique solution to (\ref{Norm:6-5}) where $\bm \psi_T=\bm \psi_\ell(T)\in L^2(\Omega)^{n-1}$.
Hence, by (\ref{Norm:6-7}), we can use the similar arguments to those used in the proof of  (\ref{Norm:6-6}) to obtain that
\begin{equation}\label{Norm:6-10}
-\langle \bm \psi_\ell(0),D \bm y_0\rangle_{L^2(\Omega)^{n-1}}=
\int_0^T \langle \chi_\omega \bm v,\chi_\omega B^\top D^\top\bm \psi_\ell\rangle_{L^2(\omega)^m}\mathrm dt.
\end{equation}
Passing to the limit for $\ell\rightarrow +\infty$ in (\ref{Norm:6-10}), and using
(\ref{Norm:6-7(1)}) and (\ref{Norm:6-9}), we get that
\begin{equation}\label{Norm:6-11}
-\langle \bm \psi^*(0),D \bm y_0\rangle_{L^2(\Omega)^{n-1}}=
\int_0^T \langle \chi_\omega \bm v,\chi_\omega B^\top D^\top\bm \psi^*\rangle_{L^2(\omega)^m}\mathrm dt.
\end{equation}
On the other hand, by  choosing $\chi_\omega B^\top D^\top \bm\psi=\chi_\omega B^\top D^\top \bm \psi^*$
in (\ref{Norm:6-3}), we find that
\begin{equation*}
\int_0^T \|\chi_\omega B^\top D^\top \bm \psi^*\|_{L^2(\omega)^m}^2\mathrm dt
=-\langle \bm\psi^*(0), D\bm y_0\rangle_{L^2(\Omega)^{n-1}}.
\end{equation*}
Finally, the above, combined with (\ref{Norm:6-11}), indicates that
\begin{equation*}
\int_0^T \|\chi_\omega B^\top D^\top \bm \psi^*\|_{L^2(\omega)^m}^2\mathrm dt
\leq \int_0^T \|\bm v\|_{L^2(\Omega)^m}^2\mathrm dt,
\end{equation*}
which, along with (\ref{Norm:6-1}),
 leads to (\ref{Yuanyuan4.21}).\\

Step 3. We show that $\bm v^*(\cdot)$ is the unique optimal control to $(NP)_T^{\bm y_0}$.

To this end, let $\widetilde{\bm v}$  be another optimal control to $(NP)_T^{\bm y_0}$. According to
 Step 2 and (ii) of Theorem~\ref{20180225-1}, $\bm v^*$ and $\widetilde{\bm v}$ are two
 optimal controls to $(\widetilde{NP})_T^{\bm y_0}$. Then we have that
\begin{equation}\label{Norm:1-7}
\|\bm v^*\|_{L^2(0,+\infty;L^2(\Omega)^m)}=\|\widetilde{\bm v}\|_{L^2(0,+\infty;L^2(\Omega)^m)}=\widetilde{N}(T,\bm y_0),
\end{equation}
\begin{equation*}
\bm v^*(t)=\widetilde{\bm v}(t)=\bm 0\;\;
\mbox{for a.e.}\;\;t>T
\end{equation*}
and
\begin{equation*}
\bm z(T;D\bm y_0,\bm v^*)=\bm z(T;D\bm y_0,\widetilde{\bm v})=\bm 0.
\end{equation*}
These imply that
\begin{equation*}
\|(\bm v^*+\widetilde{\bm v})/2\|_{L^2(0,+\infty;L^2(\Omega)^m)}\leq \widetilde{N}(T,\bm y_0),
\;\;(\bm v^*+\widetilde{\bm v})/2=\bm 0\;\;
\mbox{for a.e.}\;\;t>T,
\end{equation*}
and
\begin{equation*}
\bm z(T;D\bm y_0,(\bm v^*+\widetilde{\bm v})/2)=\bm 0.
\end{equation*}
Thus, $(\bm v^*+\widetilde{\bm v})/2$ is also an optimal control to $(\widetilde{NP})_T^{\bm y_0}$.
Then
\begin{equation*}
\|(\bm v^*+\widetilde{\bm v})/2\|_{L^2(0,+\infty;L^2(\Omega)^m)}=\widetilde{N}(T,\bm y_0).
\end{equation*}
This, together with (\ref{Norm:1-7}) and parallelogram law, implies that  $\bm v^*=\widetilde{\bm v}$.\\

Step 4. We prove that $\bm v^*(t)\not=\bm 0$ for a.e. $t\in (0,T)$.

Since the operator $\Delta-\widetilde{A}^\top$, with its domain $\left(H^2(\Omega)\cap H_0^1(\Omega)\right)^{n-1}$,
  generates an analytic semigroup on $L^2(\Omega)^{n-1}$ (see, for instance, \cite{Pazy}), the desired result
 follows from (\ref{Norm:6-1}), (i) of Lemma~\ref{Norm:5}, (\ref{20180205-2}) and (\ref{Intro:11}).\\

In summary, we finish the proof of Theorem~\ref{Norm:6}.
\hspace{65mm} $\Box$

\subsection{Continuity and monotonicity of minimal norm function}

\begin{Proposition}\label{Norm:7} Let $\bm y_0\in L^2(\Omega)^n$. The following conclusions are true:
\begin{itemize}
  \item[(i)] Suppose that $(H_1)$ holds. If $\bm y_0\in S$, then $N(T,\bm y_0)=0$ for each $T>0$
  and $\lim_{T\rightarrow +\infty} N(T,\bm y_0)=M(\bm y_0)=0$;
  If $\bm y_0\not\in S$, then the function $N(\cdot,\bm y_0)$
is strictly monotonically decreasing and right continuous over $(0,+\infty)$. Moreover, it holds that
\begin{equation}\label{20180206-2}
\lim_{T\rightarrow +\infty} N(T,\bm y_0)=M(\bm y_0)\in [0,+\infty)
\end{equation}
and
\begin{equation}\label{Norm:7-1}
\lim_{T\rightarrow 0^+} N(T,\bm y_0)=+\infty.
\end{equation}
Here, $M(\bm y_0)$ is given by (\ref{Intro:9}).
\item[(ii)] Suppose that $(H_2)$ holds. If $\bm y_0=\bm 0$, then $N(T,\bm y_0)=0$ for each $T>0$
  and $\lim_{T\rightarrow +\infty} N(T,\bm y_0)=M(\bm y_0)=0$;
  If $\bm y_0\not=\bm 0$, then the function $N(\cdot,\bm y_0)$
is strictly monotonically decreasing and right continuous over $(0,+\infty)$. Moreover, it holds that
\begin{equation*}
\lim_{T\rightarrow +\infty} N(T,\bm y_0)=M(\bm y_0)\in [0,+\infty)
\;\;\mbox{and}\;\;\lim_{T\rightarrow 0^+} N(T,\bm y_0)=+\infty.
\end{equation*}
Here, $M(\bm y_0)$ is given by (\ref{Intro:9}).
\end{itemize}
\end{Proposition}
\begin{proof} We only need to  prove the conclusion  (i), because of the reasons given in the last paragraph in subsection 3.2.

First, we suppose that $\bm y_0\in S$.
Then, by
(i) of Theorem~\ref{Norm:6},
we see that $N(T,\bm y_0)=0$ for each $T>0$. Consequently, we have that  $\lim_{T\rightarrow +\infty} N(T,\bm y_0)=M(\bm y_0)=0$.

We next suppose that  $\bm y_0\not\in S$. The rest of the proof  will be carried out by the following four steps.

\vskip 5pt

Step 1. We show that $N(\cdot,\bm y_0)$ is strictly monotonically decreasing.

Let $0<T_1<T_2<+\infty$.
By (i) of Theorem~\ref{Norm:6}, we can let  $\bm v_1^*$  be the unique optimal control to $(NP)_{T_1}^{\bm y_0}$. According to (ii) of Theorem~\ref{20180225-1}, $\bm v_1^*$ is
the optimal control to $(\widetilde{NP})_{T_1}^{\bm y_0}$. Then we have that
\begin{equation}\label{Norm:7-2}
 \|\bm v_1^*\|_{L^2(0,+\infty;L^2(\Omega)^m)}=\widetilde{N}(T_1,\bm y_0),\;\;\bm v_1^*(t)=\bm 0\;\;
 \mbox{for a.e.}\;\;t\in (T_1,+\infty)
\end{equation}
and
\begin{equation}\label{20180207-1}
\bm z(T_1;D\bm y_0,\bm v_1^*)=\bm 0.
\end{equation}
Since $T_2>T_1$, by  the second equality in (\ref{Norm:7-2}) and (\ref{20180207-1}),
we observe that
\begin{equation*}
\bm v_1^*(t)=\bm 0\;\;\mbox{for a.e.}\;\;t\in (T_2,+\infty)\;\;\mbox{and}\;\;\bm z(T_2;D\bm y_0,\bm v_1^*)=\bm 0.
\end{equation*}
These, along with the first equality in (\ref{Norm:7-2}), imply that $\bm v_1^*$ is an admissible control to
$(\widetilde{NP})_{T_2}^{\bm y_0}$ and
\begin{equation}\label{Norm:7-3}
\widetilde{N}(T_2,\bm y_0)\leq \|\bm v_1^*\|_{L^2(0,+\infty;L^2(\Omega)^m)}=\widetilde{N}(T_1,\bm y_0).
\end{equation}
We now claim  that $\widetilde{N}(T_2,\bm y_0)<\widetilde{N}(T_1,\bm y_0)$. By contradiction,
suppose that it were not true. Then by (\ref{Norm:7-3}), we would have that
$\widetilde{N}(T_1,\bm y_0)=\widetilde{N}(T_2,\bm y_0)$. This, together with
(\ref{Norm:7-2}) and the fact that $T_2>T_1$,
implies that $\bm v_1^*$ is an optimal control to $(\widetilde{NP})_{T_2}^{\bm y_0}$
and $\bm v_1^*(t)=\bm 0$
for a.e. $t\in (T_1,T_2)$. Using (ii) of Theorem~\ref{20180225-1}, we see that $\bm v_1^*$ is an optimal
control to $(NP)_{T_2}^{\bm y_0}$.
 However, by applying (i) of Theorem~\ref{Norm:6} to the problem $(NP)^{\bm y_0}_{T_2}$,
 we find that its optimal control $\bm v_1^*$ has the property: $\bm v_1^*(t)\not=\bm 0$
 for a.e. $t\in (0,T_2)$. Thus, we get a contradiction. Hence, we have that
 $\widetilde{N}(T_2,\bm y_0)<\widetilde{N}(T_1,\bm y_0)$, which
  along with (ii) of Theorem~\ref{20180225-1}, shows that $N(T_2,\bm y_0)<N(T_1,\bm y_0)$. So the function $N(\cdot,\bm y_0)$ is strictly monotonically decreasing.
 \\

Step 2. We prove that $N(\cdot,\bm y_0)$ is right continuous.

We arbitrarily fix a $T_0\in (0,+\infty)$. Let $\{T_i\}_{i\geq 1}$ be a strictly monotonically decreasing  sequence so that
$T_i\rightarrow T_0$. We aim to show that
\begin{equation}\label{Norm:7-4}
\lim_{i\to +\infty} N(T_i,\bm y_0)=N(T_0,\bm y_0).
\end{equation}
On one hand, by Step 1, we have that
\begin{equation}\label{Norm:7-5}
\lim_{i\to +\infty} N(T_i,\bm y_0)\leq N(T_0,\bm y_0).
\end{equation}
On the other hand, by the conclusion (i) of Theorem~\ref{Norm:6}, for each $i\geq 1$, $(NP)_{T_i}^{\bm y_0}$
has a   unique optimal control $\bm v_i^*$. According to (ii) of Theorem~\ref{20180225-1},
each $\bm v_i^*$ is the optimal control to the problem  $(\widetilde{NP})_{T_i}^{\bm y_0}$. Thus, we have that
\begin{equation}\label{Norm:7-6}
\|\bm v_i^*\|_{L^2(0,+\infty;L^2(\Omega)^m)}=N(T_i,\bm y_0),\;\;
\bm v_i^*(t)=\bm 0\;\;\mbox{for a.e.}\;\;t>T_i
\end{equation}
and
\begin{equation}\label{Norm:7-7}
\bm z(T_i;D\bm y_0,\bm v_i^*)=\bm 0.
\end{equation}
According to (\ref{Norm:7-5}), (\ref{Norm:7-6}),
 $L^2-$theory for parabolic system and Arzel\`{a}-Ascoli theorem, there exists
a control $\bm v_0\in L^2(0,+\infty;L^2(\Omega)^m)$
with $\bm v_0(t)=\bm 0$ for a.e. $t>T_0$ and a subsequence of $\{i\}_{i\geq 1}$, still
denoted in the same way, so that
\begin{equation}\label{Norm:7-8}
\bm v_i^*\rightarrow \bm v_0\;\;\mbox{weakly in}\;\;L^2(0,+\infty;L^2(\Omega)^m),\;\;
\|\bm v_0\|_{L^2(0,+\infty;L^2(\Omega)^m)}\leq \lim_{i\rightarrow +\infty} N(T_i,\bm y_0),
\end{equation}
and so that
\begin{equation}\label{Norm:7-9}
\bm z(\cdot;D\bm y_0,\bm v_i^*)\rightarrow \bm z(\cdot;D\bm y_0,\bm v_0)
\;\;\mbox{strongly in}\;\;C([0,T_1];L^2(\Omega)^{n-1}).
\end{equation}

It follows from (\ref{Norm:7-7}), (\ref{Norm:7-9})
and the continuity of $\bm z(\cdot;D\bm y_0,\bm v_0)$ that
\begin{eqnarray*}
\|\bm z(T_0;D\bm y_0,\bm v_0)\|_{L^2(\Omega)^{n-1}}&\leq&
\|\bm z(T_0;D\bm y_0,\bm v_0)-\bm z(T_i;D\bm y_0,\bm v_0)\|_{L^2(\Omega)^{n-1}}\\
&&+\|\bm z(T_i;D\bm y_0,\bm v_0)-\bm z(T_i;D\bm y_0,\bm v_i^*)\|_{L^2(\Omega)^{n-1}}\rightarrow 0.
\end{eqnarray*}
This implies that $\bm z(T_0;D\bm y_0,\bm v_0)=\bm 0$, i.e.,
$\bm v_0(\cdot)$
is an admissible control to the problem  $(\widetilde{NP})_{T_0}^{\bm y_0}$. So it is
an admissible control to
$(NP)_{T_0}^{\bm y_0}$ (see (ii) of Theorem~\ref{20180225-1}). Hence,
$N(T_0,\bm y_0)\leq \|\bm v_0\|_{L^2(0,+\infty;L^2(\Omega)^m)}$.
This, along with the second conclusion in (\ref{Norm:7-8}) and (\ref{Norm:7-5}), leads to
(\ref{Norm:7-4}).\\

Step 3. We show (\ref{20180206-2}).

Since $\bm y_0\not\in S$, $N(T,\bm y_0)>0$ for each $T>0$ (see (i) of Theorem~\ref{Norm:6}).
Then (\ref{20180206-2}) follows from Step 1 at once.\\

Step 4. We prove (\ref{Norm:7-1}).

By contradiction,
suppose that (\ref{Norm:7-1}) were not true. Then
 we would have that
\begin{equation*}
\lim_{T\rightarrow 0^+} N(T,\bm y_0)=\widetilde M \;\;\mbox{for some}\;\; \widetilde M\in (0,+\infty).
\end{equation*}
Let $\{T_i\}_{i\geq 1}$ be a strictly monotonically decreasing sequence satisfying $T_i\rightarrow 0$.
By (i) of Theorem~\ref{Norm:6}, for each $i\geq 1$, $(NP)_{T_i}^{\bm y_0}$ has a  unique optimal control  $\bm v_i^*$. Then by (ii) of Theorem~\ref{20180225-1},
each $\bm v_i^*$ is the optimal control to the problem  $(\widetilde{NP})_{T_i}^{\bm y_0}$.
Thus, we can use the optimality of $\bm v_i^*$ and the  conclusion of  Step 1 to obtain that
for all $i\geq 1$,
\begin{equation}\label{Norm:7-11}
\|\bm v_i^*\|_{L^2(0,+\infty;L^2(\Omega)^m)}=N(T_i,\bm y_0)\leq \widetilde M,\;\;
\bm v_i^*(t)=\bm 0\;\;\mbox{for a.e.}\;\;t>T_i,
\end{equation}
and
\begin{equation}\label{Norm:7-12}
\bm z(T_i;D\bm y_0,\bm v_i^*)=\bm 0.
\end{equation}
According to (\ref{Norm:7-11}), we can use  the same arguments as those used in Step 2 to see that
there exists a control $\bm v_0\in L^2(0,+\infty;L^2(\Omega)^m)$
and a subsequence of $\{i\}_{i\geq 1}$, still denoted in the same way, so that
\begin{equation*}
\bm v_i^*\rightarrow \bm v_0\;\;\mbox{weakly in}\;\;L^2(0,+\infty;L^2(\Omega)^m)
\end{equation*}
and so that
\begin{equation}\label{Norm:7-13}
\bm z(\cdot;D\bm y_0,\bm v_i^*)\rightarrow \bm z(\cdot;D\bm y_0,\bm v_0)\;\;
\mbox{strongly in}\;\;C([0,T_1];L^2(\Omega)^n).
\end{equation}
Meanwhile, it is clear that
\begin{eqnarray*}
&&\|\bm z(T_i;D\bm y_0,\bm v_0)\|_{L^2(\Omega)^{n-1}}\\
&\leq& \|\bm z(T_i;D\bm y_0,\bm v_0)-\bm z(T_i;D\bm y_0,\bm v_i^*)\|_{L^2(\Omega)^{n-1}}
+\|\bm z(T_i;D\bm y_0,\bm v_i^*)\|_{L^2(\Omega)^{n-1}}.
\end{eqnarray*}
Then by (\ref{Norm:7-12}) and (\ref{Norm:7-13}), we can pass to the limit for $i\rightarrow +\infty$ in the above inequality to  obtain that $D\bm y_0=\bm 0$, i.e., $\bm y_0\in S$.
This leads to a contradiction. Hence,  (\ref{Norm:7-1}) is true. \\

Thus, we finish the proof of Proposition~\ref{Norm:7}.
\end{proof}

\section{Proof of main results}
We only need to  prove Theorem~\ref{Intro:17-1}, because of the reasons given in the last paragraph in subsection 3.2. The proof of Theorem~\ref{Intro:17-1} will be organized in the
next three subsections.
\subsection{Existence and uniqueness}
\begin{Proposition}\label{Time:1}
Suppose that $(H_1)$ holds. Let  $\bm y_0\in L^2(\Omega)^n$ and let $M>0$.
The following conclusions are true:
\begin{itemize}
\item[$(i)$] If $\bm y_0\in S$, then $(TP)_M^{\bm y_0}$ has the
unique optimal control $\bm 0$ (while $0$ is the optimal time).
\item[$(ii)$] If $\bm y_0\not\in S$ and $M\leq M(\bm y_0)$, then $(TP)_M^{\bm y_0}$ has no optimal control.
\item[$(iii)$] If $\bm y_0\not\in S$ and $M> M(\bm y_0)$, then $T(M,\bm y_0)>0$ and $(TP)_M^{\bm y_0}$ has
a unique optimal control $\bm u^*$. Moreover,  $\|\bm u^*\|_{L^2(0,+\infty;L^2(\Omega)^m)}=M$.
\end{itemize}
\end{Proposition}

\begin{proof}  First of all, we recall that $\bm y_0\in S\Leftrightarrow  D\bm y_0=0$ (which follows  by
the definition of $S$ and (\ref{wang1.3}) at once). Now we  show (i).
Suppose that $\bm y_0\in S$. Then we have that  $D\bm y_0=\bm 0$. From this, one can easily check that $\bm 0$ is the unique optimal control
and $0$ is the optimal time to the problem  $(\widetilde{TP})_M^{\bm y_0}$.
Then by (i) of Theorem~\ref{20180225-1},  $\bm 0$
and $0$ are the unique optimal control and the optimal time to
$(TP)_M^{\bm y_0}$.\\

We next show the conclusion (ii). By contradiction, suppose that it were not true.
Then there would be $\bm y_0\not\in S$ and $M\leq M(\bm y_0)$ so that
$(TP)_M^{\bm y_0}$ has an optimal control $\bm{\widetilde u}$.
Thus, according to (i) of Theorem~\ref{20180225-1},
$\bm{\widetilde u}$ is an optimal control to the problem
$(\widetilde{TP})_M^{\bm y_0}$ and $\widetilde{T}(M,\bm y_0)=T(M,\bm y_0)$. These imply that
\begin{equation}\label{Time:1-a}
\|\bm{\widetilde u}\|_{L^2(0,+\infty;L^2(\Omega)^m)}\leq M,\;\;\;
\bm{\widetilde u}(t)=\bm 0\;\;\mbox{for a.e.}\;\;t>\widetilde{T}(M,\bm y_0)=T(M,\bm y_0),
\end{equation}
and
\begin{equation}\label{Time:1-b}
\bm z(T(M,\bm y_0);D\bm y_0,\bm{\widetilde u})=\bm z(\widetilde{T}(M,\bm y_0);D\bm y_0,\bm{\widetilde u})=\bm 0.
\end{equation}
Since $\bm y_0\not\in S$, we have that  $D\bm y_0\not=\bm 0$. Thus, it follows by (\ref{Time:1-b}) and (\ref{Time:1-a})
 that
$T(M,\bm y_0)>0$ and   $\bm{\widetilde u}$ is an admissible control to
the problem  $(\widetilde{NP})_{T(M,\bm y_0)}^{\bm y_0}$.
Then by making use of (ii) of Theorem~\ref{20180225-1} again, we see that  $\bm{\widetilde u}$ is an admissible control to
$(NP)_{T(M,\bm y_0)}^{\bm y_0}$, which indicates that
$$
N(T(M,\bm y_0),\bm y_0)\leq \|\bm {\widetilde u}\|_{L^2(0,+\infty;L^2(\Omega)^m)}\leq M.
$$
This, along with the strict  monotonicity of $N(\cdot,\bm y_0)$
(see (i) of Proposition~\ref{Norm:7})), yields that
\begin{equation*}
M\geq N(T(M,\bm y_0),\bm y_0)>M(\bm y_0),
\end{equation*}
which leads to a contradiction. Hence, the conclusion (ii) in this proposition is true. \\

Finally, we  prove the conclusion (iii). Suppose that  $\bm y_0\not\in S$ and $M> M(\bm y_0)$. The rest of the proof is organized by  three steps.

Step 1. We show that $(TP)_M^{\bm y_0}$ has at least one optimal control and $T(M,\bm y_0)>0$.

According to (i) of Proposition~\ref{Norm:7},
there exists $\widetilde{T}>0$ so that $M(\bm y_0)<N(\widetilde{T},\bm y_0)<M$.
By (i) of Theorem~\ref{Norm:6}, the problem $(NP)_{\widetilde T}^{\bm y_0}$ has a unique optimal control
$\widetilde{\bm u}$.
Then according to (ii) of Theorem~\ref{20180225-1},  $\widetilde{\bm u}$ is the unique optimal control to $(\widetilde{NP})_{\widetilde T}^{\bm y_0}$. Thus, by the optimality of $\widetilde{\bm u}$, we have that
\begin{equation*}
\|\widetilde{\bm u}\|_{L^2(0,+\infty;L^2(\Omega)^m)}=\widetilde{N}(\widetilde{T},\bm y_0)=N(\widetilde{T},\bm y_0)<M,
\;\;\widetilde{\bm u}(t)=\bm 0\;\;\mbox{for a.e.}\;\;t>\widetilde{T}
\end{equation*}
and
\begin{equation*}
\bm z(\widetilde{T};D\bm y_0,\widetilde{\bm u})=\bm 0.
\end{equation*}
These yield that $\widetilde{\bm u}$ is an admissible control for the problem $(\widetilde{TP})_M^{\bm y_0}$.
Hence, there exists $\{\bm u_i\}_{i\geq 1}\subseteq L^2(0,+\infty;L^2(\Omega)^m)$ and
a strictly monotonically decreasing sequence $\{T_i\}_{i\geq 1}$
 so that
\begin{equation}\label{Time:1-1}
T_i\rightarrow \widetilde{T}(M,\bm y_0),
\end{equation}
\begin{equation}\label{Time:1-2}
\|\bm u_i\|_{L^2(0,+\infty;L^2(\Omega)^m)}\leq M,\;\;\bm u_i(t)=\bm 0\;\;\mbox{for a.e.}\;\;t>T_i,
\end{equation}
and so that
\begin{equation}\label{Time:1-3}
\bm z(T_i;D\bm y_0,\bm u_i)=\bm 0.
\end{equation}

By (\ref{Time:1-1})-(\ref{Time:1-3}), we can use the similar arguments to those
used in  Step 2 of the proof of  Proposition~\ref{Norm:7} to obtain  $\bm {\widetilde u^*}\in L^2(0,+\infty;L^2(\Omega)^m)$ so that
\begin{equation}\label{Time:1-4}
\|\bm {\widetilde u^*}\|_{L^2(0,+\infty;L^2(\Omega)^m)}\leq M,\;\;
\bm {\widetilde u^*}(t)=\bm 0\;\;\mbox{for a.e.}\;\;t>\widetilde{T}(M,\bm y_0),
\end{equation}
and so that
\begin{equation}\label{Time:1-5}
\bm z(\widetilde{T}(M,\bm y_0);D\bm y_0,\bm {\widetilde u^*})=\bm 0.
\end{equation}
Since $\bm y_0\not\in S$, i.e., $D\bm y_0\not=\bm 0$, it follows from (\ref{Time:1-5}) and (\ref{Time:1-4})
 that $\widetilde{T}(M,\bm y_0)>0$ and
$\bm {\widetilde u^*}$ is an optimal control to $(\widetilde{TP})_M^{\bm y_0}$.
These, along with (i) of Theorem~\ref{20180225-1}, show that $\bm {\widetilde u^*}$ is an optimal control to $(TP)_M^{\bm y_0}$ and
$T(M,\bm y_0)=\widetilde{T}(M,\bm y_0)>0$.\\

Step 2. We prove that
any optimal control $\bm u^*$ to $(TP)_M^{\bm y_0}$ satisfies that
$\|\bm u^*\|_{L^2(0,+\infty;L^2(\Omega)^m)}\\=M$.

By contradiction, suppose that it were not true. Then  $(TP)_M^{\bm y_0}$ would have an optimal control $\bm u^*$
so that
\begin{equation}\label{Time:2-1}
\|\bm u^*\|_{L^2(0,+\infty;L^2(\Omega)^m)}\triangleq \widetilde{M}<M.
\end{equation}
According to (i) of Theorem~\ref{20180225-1}, $\bm u^*$ is an optimal control to $(\widetilde{TP})_M^{\bm y_0}$ and $\widetilde{T}(M,\bm y_0)
=T(M,\bm y_0)$. Then by the optimality of $\bm u^*$, we get that
\begin{equation}\label{Time:2-3}
\bm z(\widetilde{T}(M,\bm y_0);D\bm y_0,\bm u^*)=\bm 0\;\;\mbox{and}\;\;\bm u^*(t)=\bm 0\;\;\mbox{for a.e.}\;\;
t>\widetilde{T}(M,\bm y_0).
\end{equation}
We denote $\bm z^*(\cdot)\triangleq \bm z(\cdot;D\bm y_0,\bm u^*)$.
Since $\widetilde{T}(M,\bm y_0)=T(M,\bm y_0)>0$ (see Step 1) and
$\bm z^*\in C([0,\widetilde{T}(M,\bm y_0)];L^2(\Omega)^{n-1})$, we see that for any $\varepsilon>0$
(it will be precised later), there is a constant
$\delta\triangleq \delta(\varepsilon)\in (0,\widetilde{T}(M,\bm y_0)/2)$ so that
\begin{equation}\label{Time:2-4}
\|D\bm y_0-\bm z^*(\delta)\|_{L^2(\Omega)^{n-1}}<\varepsilon.
\end{equation}
According to (\ref{Preli:5-1}) in Proposition~\ref{Preli:5},
and the equivalence between the observability  and the null controllability (see, for instance,  \cite{Coron}),
 there is a control
$\bm v_0\in L^2(0,+\infty;L^2(\Omega)^m)$ so that
\begin{equation}\label{20180206-3}
\bm v_0(t)=\bm 0\;\;\mbox{for a.e.}\;\;t>\widetilde{T}(M,\bm y_0)/2;
\end{equation}
\begin{equation}\label{Time:2-5}
\|\bm v_0\|_{L^2(0,+\infty;L^2(\Omega)^m)}\leq C(\widetilde{T}(M,\bm y_0))
\|D\bm y_0-\bm z^*(\delta)\|_{L^2(\Omega)^{n-1}};
\end{equation}
and so that $\bm w\in C([0,+\infty);L^2(\Omega)^{n-1})$ satisfies
\begin{equation}\label{Time:2-6}
\left\{
\begin{array}{lll}
\bm w_t-\Delta \bm w+\widetilde{A}\bm w=\chi_\omega D B\bm v_0&\mbox{in}&\Omega\times (0,+\infty),\\
\bm w=\bm 0&\mbox{on}&\partial \Omega\times (0,+\infty),\\
\bm w(0)=D\bm y_0-\bm z^*(\delta)&\mbox{in}&\Omega
\end{array}
\right.
\end{equation}
and
\begin{equation}\label{Time:2-7}
\bm w(t)=\bm 0\;\;\mbox{for all}\;\;t\geq \widetilde{T}(M,\bm y_0)/2.
\end{equation}

Denote $\bm z^*_\delta(\cdot)\triangleq \bm z^*(\cdot+\delta)$. On one hand, by
 (\ref{Time:2-6}), the first conclusion in (\ref{Time:2-3}) and
(\ref{Time:2-7}), we have that
\begin{equation}\label{Time:2-10}
\left\{
\begin{array}{lll}
(\bm z_\delta^*+\bm w)_t-\Delta (\bm z_\delta^*+\bm w)
+\widetilde{A}(\bm z_\delta^*+\bm w)&&\\
\;\;\;\;\;\;\;\;\;\;\;\;\;\;\;\;\;\;\;\;\;\;\;\;\;\;\;\;\;\;\;\;\;\;\;\;\;\;\;\;\;\;
=\chi_\omega D B(\bm u^*(t+\delta)+\bm v_0)&\mbox{in}&\Omega\times (0,+\infty),\\
\bm z_\delta^*+\bm w=\bm 0&\mbox{on}&\partial \Omega\times (0,+\infty),\\
(\bm z_\delta^*+\bm w)(0)=D\bm y_0&\mbox{in}&\Omega
\end{array}
\right.
\end{equation}
and
\begin{equation}\label{Time:2-11}
(\bm z_\delta^*+\bm w)(\widetilde{T}(M,\bm y_0)-\delta)=\bm 0.
\end{equation}
On the other hand, we choose $\varepsilon=\frac{M-\widetilde{M}}{2(C(T(M,\bm y_0))+1)}$.
It follows from (\ref{Time:2-1}), (\ref{Time:2-5}) and (\ref{Time:2-4}) that
\begin{equation}\label{20180226-1}
\|\bm u^*(\cdot+\delta)+\bm v_0(\cdot)\|_{L^2(0,+\infty;L^2(\Omega)^m)}\leq M.
\end{equation}
Noting that $\bm u^*(t+\delta)+\bm v_0(t)=\bm 0$ for a.e. $t>\widetilde{T}(M,\bm y_0)-\delta$
(which follows by the second conclusion of (\ref{Time:2-3}) and (\ref{20180206-3})),
by (\ref{Time:2-10})-(\ref{20180226-1}) and the optimality of $\widetilde{T}(M,\bm y_0)$, we obtain
that $\widetilde{T}(M,\bm y_0)-\delta\geq \widetilde{T}(M,\bm y_0)$,
which leads to a contradiction.\\

Step 3. We show that $(TP)_M^{\bm y_0}$ has a unique optimal control.

The result follows from Steps 1-2 and similar arguments as those in Step 3 of the proof of Theorem~\ref{Norm:6}.\\

Hence, according to Steps 1-3, the conclusions in (iii) are true.\\

In summary, we finish the proof of this proposition.
\end{proof}

\begin{Corollary}\label{Time:4}
Suppose that $(H_1)$ holds and $\bm y_0\not\in S$.
The function $T(\cdot,\bm y_0)$ is strictly monotonically decreasing over $(M(\bm y_0),+\infty)$.
\end{Corollary}
\begin{proof} We arbitrarily fix $M_1, M_2\in (M(\bm y_0),+\infty)$ so that $M_2>M_1$.
By (iii) of Proposition~\ref{Time:1}, we have that
$T(M_1,\bm y_0)>0$ and $(TP)_{M_1}^{\bm y_0}$ has a unique optimal control $\bm u_1^*$  so that
\begin{equation}\label{Time:4-1}
\|\bm u_1^*\|_{L^2(0,+\infty;L^2(\Omega)^m)}=M_1<M_2.
\end{equation}
Thus, according to (i) of
Theorem~\ref{20180225-1},  $\bm u_1^*$ is the optimal control to the problem $(\widetilde{TP})_{M_1}^{\bm y_0}$. Then by the optimality of $\bm u_1^*$, we obtain that
\begin{equation}\label{20180206-4}
\bm z(\widetilde{T}(M_1,\bm y_0);D\bm y_0,\bm u_1^*)=\bm 0\;\;\mbox{and}\;\;\bm u_1^*(t)=\bm 0\;\;
\mbox{for a.e.}\;\;t>\widetilde{T}(M_1,\bm y_0).
\end{equation}
It follows from (\ref{Time:4-1}) and (\ref{20180206-4}) that $\bm u_1^*$ is
an admissible control to $(\widetilde{TP})_{M_2}^{\bm y_0}$.
Hence, $\widetilde{T}(M_2,\bm y_0)\leq \widetilde{T}(M_1,\bm y_0)$.

If $\widetilde{T}(M_2,\bm y_0)=\widetilde{T}(M_1,\bm y_0)$,
then by (\ref{Time:4-1}) and (\ref{20180206-4}), we would have  that
$\bm u_1^*$ is an optimal control to $(\widetilde{TP})_{M_2}^{\bm y_0}$.
Thus, according to (i) of Theorem~\ref{20180225-1},
$\bm u_1^*$ is an optimal control to $(TP)_{M_2}^{\bm y_0}$.
This, along with (iii) of Proposition~\ref{Time:1}, yields that $\|\bm u_1^*\|_{L^2(0,+\infty;L^2(\Omega)^m)}=M_2$,
which contradicts  (\ref{Time:4-1}). Thus, $\widetilde{T}(M_2,\bm y_0)<\widetilde{T}(M_1,\bm y_0)$.
Using (i) of Theorem~\ref{20180225-1} again, we obtain that $T(M_2,\bm y_0)<T(M_1,\bm y_0)$.

This completes the proof.
\end{proof}

\subsection{Equivalence between minimal time and minimal norm control problems}
\begin{Proposition}\label{Equivalence:1} Suppose that
$(H_1)$ holds and $\bm y_0\not\in S$.
For each $T>0$, $T(N(T,\bm y_0),\bm y_0)=T$ and the optimal control to $(NP)_T^{\bm y_0}$ is the optimal control
to $(TP)_{N(T,\bm y_0)}^{\bm y_0}$. Conversely, for each $M>M(\bm y_0)$, $N(T(M,\bm y_0),\bm y_0)=M$
and the optimal control to $(TP)_M^{\bm y_0}$ is the optimal control to $(NP)_{T(M,\bm y_0)}^{\bm y_0}$.
\end{Proposition}
\begin{proof}
Let $T>0$. According to (i) of Theorem~\ref{Norm:6},
 $(NP)_T^{\bm y_0}$ has a unique optimal control $\bm u_1^*$ so that
\begin{equation}\label{Main:1}
\|\bm u_1^*\|_{L^2(0,+\infty;L^2(\Omega)^m)}=N(T,\bm y_0)>0
\;\;\mbox{and}\;\;\bm u_1^*(t)=\bm 0\;\;\mbox{for a.e.}\;\;t>T.
\end{equation}
Thus, it follows by (ii) of Theorem~\ref{20180225-1} that $\bm u_1^*$ is the optimal control to $(\widetilde{NP})_T^{\bm y_0}$. Then by the optimality of $\bm u_1^*$, we have that
\begin{equation}\label{Main:2}
\bm z(T;D\bm y_0,\bm u_1^*)=\bm 0.
\end{equation}
It follows from (\ref{Main:1}) and (\ref{Main:2}) that $\bm u_1^*$ is an admissible control to
$(\widetilde{TP})_{N(T,\bm y_0)}^{\bm y_0}$
and $\widetilde{T}(N(T,\bm y_0),\bm y_0)\leq T$. We now claim that
\begin{equation}\label{Main:3}
\widetilde{T}(N(T,\bm y_0),\bm y_0)=T.
\end{equation}
If (\ref{Main:3}) is proved, then according to (i) of Theorem~\ref{20180225-1},
\begin{equation}\label{20180226-2}
T(N(T,\bm y_0),\bm y_0)=T.
\end{equation}
By contradiction, suppose that (\ref{Main:3}) were not true. Then we would have that
$\widetilde{T}(N(T,\bm y_0),\\ \bm y_0)<T$. By (i) of
Proposition~\ref{Norm:7}, we get that
\begin{equation}\label{Main:3-1}
N(T,\bm y_0)>M(\bm y_0).
\end{equation}
This, along with (iii) of Proposition~\ref{Time:1},
implies that $(TP)_{N(T,\bm y_0)}^{\bm y_0}$ has a unique optimal control $\bm u_2^*$ so that
\begin{equation}\label{20180206-5}
\|\bm u_2^*\|_{L^2(0,+\infty;L^2(\Omega)^m)}=N(T,\bm y_0).
\end{equation}
Thus, according to (i) of Theorem~\ref{20180225-1}, $\bm u_2^*$ is the optimal control to $(\widetilde{TP})_{N(T,\bm y_0)}^{\bm y_0}$.
 Then, by the optimality of $\bm u_2^*$, we obtain that
\begin{equation}\label{20180206-6}
\bm z(\widetilde{T}(N(T,\bm y_0),\bm y_0);D\bm y_0,\bm u_2^*)=\bm 0\;\;
\mbox{and}\;\;\bm u_2^*(t)=\bm 0\;\;
\mbox{for a.e.}\;\;t>\widetilde{T}(N(T,\bm y_0),\bm y_0).
\end{equation}
Meanwhile, since $T>\widetilde{T}(N(T,\bm y_0),\bm y_0)$, by  (\ref{20180206-6}), we see that
\begin{equation}\label{20180226-3}
\bm z(T;D\bm y_0,\bm u_2^*)=\bm 0\;\;\mbox{and}\;\;\bm u_2^*(t)=\bm 0\;\;\mbox{for a.e.}\;\;t>T.
\end{equation}
Since $N(T,\bm y_0)=\widetilde{N}(T,\bm y_0)$ (see (ii) of Theorem~\ref{20180225-1}), by (\ref{20180226-3})
and (\ref{20180206-5}), we observe that
\begin{equation*}
\bm u_2^*\;\;\mbox{is an optimal control to}\;\;(\widetilde{NP})_T^{\bm y_0}.
\end{equation*}
From the latter, (ii) of Theorem~\ref{20180225-1} and the second relation of (\ref{20180206-6})
 it follows that $\bm u_2^*$ is an optimal control to $(NP)_T^{\bm y_0}$ and
 $\bm u_2^*(t)=\bm 0$ for a.e. $t\in (\widetilde{T}(N(T,\bm y_0),\bm y_0),T)$.
 However, by applying (i) of Theorem~\ref{Norm:6} to the problem $(NP)_T^{\bm y_0}$,
 we find that its optimal control $\bm u_2^*$ has the property that $\bm u_2^*(t)\not=\bm 0$ for
 a.e. $t\in (0,T)$. Thus, we get a contradiction. Hence, (\ref{Main:3}) follows.
Furthermore, by  (\ref{Main:1})-(\ref{Main:3}), we see that
$\bm u_1^*$ is an optimal control to $(\widetilde{TP})_{N(T,\bm y_0)}^{\bm y_0}$.
Thus, by (i) of Theorem~\ref{20180225-1}, (\ref{Main:3-1}) and (iii) of Proposition~\ref{Time:1},
we conclude that $\bm u_1^*$ is the optimal control to $(TP)_{N(T,\bm y_0)}^{\bm y_0}$.

Conversely, for each $M>M(\bm y_0)$, according to (iii) of Proposition~\ref{Time:1},
$T(M,\bm y_0)>0$ and
$(TP)_M^{\bm y_0}$ has a unique optimal control  $\bm u_3^*$ so that
\begin{equation}\label{Main:4}
\|\bm u_3^*\|_{L^2(0,+\infty;L^2(\Omega)^m)}=M.
\end{equation}
Thus, $\bm u_3^*$ is the optimal control to $(\widetilde{TP})_M^{\bm y_0}$ (see (i) of
Theorem~\ref{20180225-1}). By the optimality of $\bm u_3^*$, we get that
\begin{equation}\label{Main:5}
\bm z(\widetilde{T}(M,\bm y_0);D\bm y_0,\bm u_3^*)=\bm 0\;\;\mbox{and}\;\;\bm u_3^*(t)=\bm 0\;\;
\mbox{for a.e.}\;\;t>\widetilde{T}(M,\bm y_0).
\end{equation}
Since $\widetilde{T}(M,\bm y_0)=T(M,\bm y_0)$ (see (i) of Theorem~\ref{20180225-1}), it follows from
(\ref{Main:4}) and (\ref{Main:5}) that
 $\bm u_3^*$ is an admissible control to $(\widetilde{NP})_{T(M,\bm y_0)}^{\bm y_0}$ and
\begin{equation}\label{20180226-4}
\widetilde{N}(T(M,\bm y_0),\bm y_0)\leq \|\bm u_3^*\|_{L^2(0,+\infty;L^2(\Omega)^m)}=M.
\end{equation}
This, along with (ii) of Theorem~\ref{20180225-1}, implies that $N(T(M,\bm y_0),\bm y_0)\leq M$.
We next claim that
\begin{equation}\label{Main:6}
N(T(M,\bm y_0),\bm y_0)=M.
\end{equation}
By contradiction, suppose that (\ref{Main:6}) were not true. Then we would have that $N(T(M,\bm y_0),\\ \bm y_0)<M$.
Noting that $T(M,\bm y_0)>0$,
 by (\ref{20180226-2}), (\ref{Main:3-1}) and Corollary~\ref{Time:4},
we get that
\begin{equation*}
T(M,\bm y_0)=T(N(T(M,\bm y_0),\bm y_0),\bm y_0)>T(M,\bm y_0),
\end{equation*}
which leads to a contradiction. Hence, (\ref{Main:6}) follows. It follows from (\ref{Main:6})
and (ii) of Theorem~\ref{20180225-1} that
\begin{equation}\label{20180227-1}
\widetilde{N}(T(M,\bm y_0),\bm y_0)=M.
\end{equation}
Since $\bm u_3^*$ is
an admissible control to $(\widetilde{NP})_{T(M,\bm y_0)}^{\bm y_0}$, by (\ref{20180226-4}) and (\ref{20180227-1}), we observe that
$\bm u_3^*$ is an optimal control to $(\widetilde{NP})_{T(M,\bm y_0)}^{\bm y_0}$. This implies that
$\bm u_3^*$ is the optimal control to $(NP)_{T(M,\bm y_0)}^{\bm y_0}$ (see (ii) of Theorem~\ref{20180225-1} and
(i) of Theorem~\ref{Norm:6}).

\vskip 5pt

Thus, we finish the proof of Proposition \ref{Equivalence:1}.
\end{proof}

\subsection{End of the proof}

According to Proposition~\ref{Time:1}, it suffices
to show (ii) in Theorem~\ref{Intro:17-1}.
Let $T^*$ and $\bm u^*$ be the optimal time and the optimal control to $(TP)_M^{\bm y_0}$.
Since $T^*=T(M,\bm y_0)>0$ (see (iii) of Proposition~\ref{Time:1}), it follows from Proposition~\ref{Equivalence:1}
that $\bm u^*$ is the optimal control to $(NP)_{T^*}^{\bm y_0}$ and $M=N(T^*,\bm y_0)$. These,
along with (i) of Theorem~\ref{Norm:6}, imply (\ref{Intro:18}) and (\ref{Intro:19}).

Conversely, suppose that (\ref{Intro:18}) and (\ref{Intro:19}) hold.
It is obvious that $T^*>0$. By (i) of Theorem~\ref{Norm:6}, (\ref{Intro:18}), (\ref{Intro:19}) and
(i) of Proposition~\ref{Norm:7},
we see that
\begin{equation}\label{Main:9}
\bm u^*(\cdot)\;\;\mbox{is the optimal control to}\;\;(NP)_{T^*}^{\bm y_0}
\;\;\mbox{and}\;\;M=N(T^*,\bm y_0)>M(\bm y_0).
\end{equation}
Since $M=N(T(M,\bm y_0),\bm y_0)$ (see Proposition~\ref{Equivalence:1}), it follows from (i) of Proposition~\ref{Norm:7}
and the second relation of (\ref{Main:9}) that $T^*=T(M,\bm y_0)$. This, together with
Proposition~\ref{Equivalence:1} and  (\ref{Main:9}), yields that $\bm u^*(\cdot)$ is the optimal control
to $(TP)_M^{\bm y_0}$.

In summary, we finish the proof of Theorem~\ref{Intro:17-1}.
\hspace{60mm} $\Box$

\vskip 8pt

\noindent {\bf Acknowlegements.} The authors gratefully thank Professor Gengsheng Wang for his valuable suggestions.



\end{document}